\documentclass[10pt]{article}

\usepackage[margin=1.5in,left=1.5in]{geometry}
\usepackage[default,regular,semibold]{sourceserifpro}
\usepackage[T1]{fontenc}

\usepackage{amsmath,amsthm,mathtools,amssymb, amscd}
\usepackage{enumitem}
\usepackage{graphicx}
\usepackage{subcaption}
\usepackage{booktabs}
\usepackage{multirow}
\usepackage[all,cmtip]{xy}

\usepackage[backend=biber]{biblatex}
\usepackage{xcolor}
\usepackage{hyperref}
\definecolor{darkred}{rgb}{0.6, 0, 0}
\hypersetup{
  colorlinks=true,
  linkcolor=black,
  citecolor=darkred,
  urlcolor=blue
}

\usepackage{fancyhdr}
\fancypagestyle{plain}{%
  \fancyhf{}\fancyfoot[C]{\thepage}%
}

\usepackage{tikz-cd}
\usetikzlibrary{positioning}
\usetikzlibrary{arrows}
\usetikzlibrary{arrows.meta}
\usepackage{algorithm}
\usepackage{algpseudocode}
\usepackage{enumitem}
\algnewcommand\algorithmicclass{\textbf{class}}
\algblockdefx{Class}{EndClass}
  [1]{\algorithmicclass\ #1}
  {\textbf{end class}}

\newcommand{\Rips}{\mathrm{Rips}}

\newcommand{\diam}{\mathrm{diam}}
\newcommand{\Nerve}{\mathrm{Nrv}}
\newcommand{\CoNerve}{\mathrm{CoNrv}}

\newtheorem{definition}{Definition}[section]
\newtheorem{theorem}[definition]{Theorem}
\newtheorem{lemma}[definition]{Lemma}
\newtheorem{proposition}[definition]{Proposition}
\newtheorem{corollary}[definition]{Corollary}
\newtheorem{remark}[definition]{Remark}
\newtheorem{example}[definition]{Example}

\newcommand{\Cat}[1]{\boldsymbol{\mathbf{#1}}}

\newcommand{\poset}{(\mathbb{R}_+,\leq)}

\usepackage{tcolorbox}
\definecolor{mySubtleGray}{RGB}{240,240,240}
\definecolor{clr_green}{RGB}{20,101,93}
\definecolor{clr_red}{RGB}{230,57,70}
\definecolor{clr_blue}{RGB}{39,125,161}
\definecolor{clr_orange}{RGB}{247,127,0}
\definecolor{clr_yellow}{RGB}{252,191,73}
\definecolor{clr_purple}{RGB}{77,90,175}
\tcbuselibrary{skins,breakable,hooks}
\newtcolorbox{note}[1][]{
  breakable,
  colback=mySubtleGray,
  colframe=mySubtleGray,
  boxrule=0pt,
  arc=1mm,
  outer arc=1mm,
  left=10pt, right=10pt, top=8pt, bottom=8pt,
  before skip=\medskipamount, after skip=\medskipamount,
  before upper={%
    \if\relax\detokenize{#1}\relax
    \else
      \noindent\textbf{#1}\par\medskip
    \fi
  }
}

\usepackage{authblk}
\author{Ant\'onio Leit\~ao\\ \texttt{antonio.lagedesousaleitao@sns.it}}
\affil{Scuola Normale Superiore, Pisa, Italy \\ DataShape, Inria-Saclay, France}
\title{\textbf{It's all about covers}\\Persistent Homology of Cover Refinements}

\begin{document}
\maketitle
\begin{abstract}
     The computational cost of persistent homology is often dominated by the growth of the underlying simplicial filtrations.
     Many different filtrations exist, each with its own assumptions and trade-offs, but all face some form of this growth which can be exponential in the worst case, as for the Vietoris–Rips.
     We recast this problem at the level of covers, developing a framework in which filtrations and persistence modules can be constructed, analyzed, and compared through simple relations between covers rather than at the level of simplicial complexes.
     The guarantees propagate through any functor that preserves the contiguity of refinement maps; we give the example of two such functors: the Nerve and the Co-Nerve.
     Working at this level is drastically simpler, with stronger, more general consequences.
     We explore this perspective and show how it can be used to construct a robust approximation of the Vietoris-Rips filtration that is orders of magnitude smaller, while maintaining a $\log 3$-interleaving unconditionally for any metric space.
     The resulting filtration restores near-linear scaling in the number of data points and enables to efficiently capture homology at high degrees.
\end{abstract}


\section{Introduction}

Topological Data Analysis (TDA) uses methods of topology---the field of mathematics that studies qualitative geometric properties---to extract qualitative information from data~\cite{carlsson2009topology}.
This type of analysis has found wide application in diverse areas~\cite{DONUT}.
One of the main tools of TDA is persistent homology, which assigns to data a nested sequence of topological spaces indexed over a parameter; called a \textit{filtration}.
For each parameter value the assigned topological space can be seen as a representation, in some appropriate sense, of the data at that scale.
The output of persistent homology is a summary of all the relevant topological features, such as holes and connected components, along with the scale over which they remain relevant: their persistence.
Because these sequences are nested, the topological spaces grow with the parameter value, exponentially in the worst case.
This has confined the computation of persistent homology to moderate-sized datasets and low homological degree.

There is a diverse range of methods and simplicial complexes that are used in Topological Data Analysis~\cite{Silva2004TopologicalEU,edelsbrunner2003shape}.
Some have been developed specifically to curb this growth in filtration sizes~\cite{sheehy2012linear, graf2026floodcomplexlargescalepersistent, barmak2012strong, boissonnat2019computing, dlotko2013simplification}.
While each achieves significant results, they typically rely on assumptions that constrain the geometry, dimension, or homological degree.

Here we take a step back and create at the cover level the necessary framework that allows us to tackle this issue.
Our central observation is that covers are the natural level of abstraction for creating simplicial filtrations and comparing their persistence modules.
The key property is that all refinement maps between two covers are contiguous, a discrete analogue of homotopic maps, and induce contiguous simplicial maps.
As a consequence, an interleaving between covers induces an interleaving between the persistence modules of any combination of simplicial complexes built from them.
Working at the cover level is drastically simpler and has stronger, more general consequences.

We apply this framework to construct a robust approximation of the Vietoris–Rips filtration.
We introduce an opposing force to the growth of covers, a persistent partition that progressively coarsens the space as the scale
increases.
By quotienting the covers with this partition at each scale, we obtain a sequence of smaller covers whose associated simplicial complexes are $\log 3$-interleaved with the Vietoris–Rips filtration.
This guarantee is unconditional: it does not depend on doubling dimension, ambient Euclidean structure, or homological degree, and holds for any metric space.
The resulting filtration is orders of magnitude smaller than the standard Vietoris–Rips, restoring near-linear scaling in the number of data points.

\subsection{Contributions}

\begin{enumerate}
    \item \textbf{Covers as the origin of topological structure.}
Given a cover $\mathcal{U} = \{U_i\}_{i \in I}$ of a space $X$, there are two simplicial complexes that have classically been built from it: the Nerve, whose simplices correspond to sets with non-empty intersection, and the Co-Nerve, whose simplices are subsets of $X$ contained in a single cover element.
The Dowker theorem~\cite{dowker1952homology} guarantees that these two complexes have isomorphic homology (Figure~\ref{fig:nerve_conerve}).

We show that proving an interleaving between the persistence modules of any combination of Nerves and Co-Nerve reduces to showing the existence of refinement maps between the underlying covers at each scale (Theorem~\ref{thm:main_propagation}).
Since all refinement maps between two covers are contiguous (Proposition~\ref{prop:all_refinements_contiguous}).
As a first application, we recover the classical interleaving between the \v{C}ech and Vietoris-Rips filtrations (Corollary~\ref{cor:cech_rips}) as a consequence of a $2$-approximation between their underlying ball and maximal-clique covers.
\newsavebox{\ExampleCover}
\savebox{\ExampleCover}{
\begin{tikzpicture}[
  clique_cover/.style={fill=color_r20g101b93, opacity=0.2},
  edge65/.style={color_r194g202b201, opacity=0.8, thick},
  edge66/.style={black, opacity=0.8, thick},
  node/.style={circle, fill=black, inner sep=1.50pt},scale=0.7]
  \definecolor{color_r194g202b201}{rgb}{0.761,0.792,0.788}
  \definecolor{color_r20g101b93}{rgb}{0.078,0.396,0.365}
  \fill[clique_cover] (2.000,1.600) circle (0.881);
  \fill[clique_cover] (2.850,1.400) circle (0.632);
  \fill[clique_cover] (1.500,2.833) circle (0.784);
  \fill[clique_cover] (2.500,3.500) circle (0.639);
  \fill[clique_cover] (3.400,3.350) circle (0.632);
  \fill[clique_cover] (3.500,2.400) circle (0.771);
  \node[node] at (1.000,3.000) {};
  \node[node] at (2.000,3.300) {};
  \node[node] at (1.500,2.200) {};
  \node[node] at (3.000,3.700) {};
  \node[node] at (3.800,3.000) {};
  \node[node] at (3.200,1.800) {};
  \node[node] at (2.500,1.000) {};
\end{tikzpicture}
}

\newsavebox{\ExampleCoNerve}
\savebox{\ExampleCoNerve}{
\begin{tikzpicture}[
  edge19/.style={black, opacity=0.8, thick},
  node/.style={circle, fill=black, inner sep=1.50pt},
  triangle/.style={fill=color_r230g57b71, opacity=0.3, draw=none},scale=0.7]
  \definecolor{color_r230g57b71}{rgb}{0.902,0.224,0.278}
  \fill[triangle] (1.000,3.000) -- (2.000,3.300) -- (1.500,2.200) -- cycle;
  \draw[edge19] (1.000,3.000) -- (2.000,3.300);
  \draw[edge19] (1.000,3.000) -- (1.500,2.200);
  \draw[edge19] (2.000,3.300) -- (1.500,2.200);
  \draw[edge19] (2.000,3.300) -- (3.000,3.700);
  \draw[edge19] (1.500,2.200) -- (2.500,1.000);
  \draw[edge19] (3.000,3.700) -- (3.800,3.000);
  \draw[edge19] (3.800,3.000) -- (3.200,1.800);
  \draw[edge19] (3.200,1.800) -- (2.500,1.000);
  \node[node] at (1.000,3.000) {};
  \node[node] at (2.000,3.300) {};
  \node[node] at (1.500,2.200) {};
  \node[node] at (3.000,3.700) {};
  \node[node] at (3.800,3.000) {};
  \node[node] at (3.200,1.800) {};
  \node[node] at (2.500,1.000) {};
\end{tikzpicture}
}

\newsavebox{\ExampleNerve}
\savebox{\ExampleNerve}{
\begin{tikzpicture}[
 point/.style={circle, fill=black, inner sep=1.5pt},
 connection/.style={thick, opacity=0.8},
 scale=0.3
]
  \node[point] (A) at (90:3.5)  {};
  \node[point] (B) at (150:3.5) {};
  \node[point] (C) at (210:3.5) {};
  \node[point] (D) at (270:3.5) {};
  \node[point] (E) at (330:3.5) {};
  \node[point] (F) at (30:3.5)  {};

  \draw[connection] (A) -- (B);
  \draw[connection] (B) -- (C);
  \draw[connection] (C) -- (D);
  \draw[connection] (D) -- (E);
  \draw[connection] (E) -- (F);
  \draw[connection] (F) -- (A);

\end{tikzpicture}
}

\begin{figure}[htb]
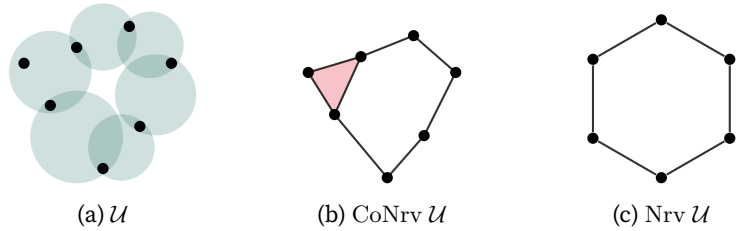

\centering

\begin{subfigure}[b]{0.26\textwidth}
\centering
\usebox{\ExampleCover}
\caption{$\mathcal{U}$}
\end{subfigure}
\begin{subfigure}[b]{0.26\textwidth}
\centering
\usebox{\ExampleCoNerve}
\caption{$\CoNerve\;\mathcal{U}$}
\end{subfigure}
\begin{subfigure}[b]{0.26\textwidth}
\centering
\usebox{\ExampleNerve}
\caption{$\Nerve\;\mathcal{U}$}
\end{subfigure}

\caption{Illustration of a cover (a) and the resulting simplicial complexes obtained from the Co-Nerve (b) and the Nerve (c). The Dowker theorem~\cite{dowker1952homology} guarantees the existence of an isomorphism between their homology.}
\label{fig:nerve_conerve}

\end{figure} 
\vspace{0.3cm}
\item 
\textbf{Quotient before complex.}
With the objective of reducing the size of filtrations, we show (Corollary~\ref{cor:quotient_of_covers}) that instead of taking the quotient of a simplicial complex, one can quotient the cover first and then build the simplicial complex on the smaller space (Fig.~\ref{fig:quotient_complexes}):
\begin{align*}
    \CoNerve(\mathcal{U})/P_X &= \CoNerve(\mathcal{U}/P_X), \\
    \Nerve(\mathcal{U})/P_I &= \Nerve(\mathcal{U}/P_I).
\end{align*}
This is particularly valuable when building the simplicial complex is itself expensive, as is the case for clique complexes.
Rather than constructing a large complex and then reducing it, we build directly on the reduced space.
Combined with the previous result, any approximation with a quotient cover induces an approximation between the persistence modules of any combination of Nerve and Co-Nerve built from them.
\begin{figure}[htb]
    \centering
    \newsavebox{\QPanelCov}
\savebox{\QPanelCov}{

\begin{tikzpicture}[
  clique_cover/.style={fill=color_r20g101b93, opacity=0.2},
  edge65/.style={color_r194g202b201, opacity=0.8, thick},
  edge66/.style={black, opacity=0.8, thick},
  partition/.style={black, opacity=0.8, thick, dashed, rounded corners=8pt},
  node/.style={circle, fill=black, inner sep=1.50pt},scale=0.7]
  
  \definecolor{color_r194g202b201}{rgb}{0.761,0.792,0.788}
  \definecolor{color_r20g101b93}{rgb}{0.078,0.396,0.365}
  
  \fill[clique_cover] (2.000,1.600) circle (0.881);
  \fill[clique_cover] (2.850,1.400) circle (0.632);
  \fill[clique_cover] (1.500,2.833) circle (0.784);
  \fill[clique_cover] (2.500,3.500) circle (0.639);
  \fill[clique_cover] (3.400,3.350) circle (0.632);
  \fill[clique_cover] (3.500,2.400) circle (0.771);

  \draw[partition] (0.600, 3.100) -- (2.400, 3.600) -- (1.500, 1.700) -- cycle;
  
  \draw[partition] (2.550, 3.600) -- (3.150, 4.200) -- (4.250, 3.100) -- (3.650, 2.500) -- cycle;
  
  \draw[partition] (3.050, 2.250) -- (3.650, 1.650) -- (2.650, 0.550) -- (2.050, 1.150) -- cycle;

  \node[node] at (1.000,3.000) {};
  \node[node] at (2.000,3.300) {};
  \node[node] at (1.500,2.200) {};
  \node[node] at (3.000,3.700) {};
  \node[node] at (3.800,3.000) {};
  \node[node] at (3.200,1.800) {};
  \node[node] at (2.500,1.000) {};
  
  \node[above] at (1.000, 3.500) {$C_1$};
  \node[above right] at (3.800, 3.800) {$C_2$};
  \node[below right] at (3.300, 1.500) {$C_3$};
\end{tikzpicture}
}

\newsavebox{\QPanelQ}
\savebox{\QPanelQ}{
\begin{tikzpicture}[
scale=0.4,
  ball_cover/.style={fill=color_r20g101b93, opacity=0.2},
  node/.style={circle, fill=black, inner sep=1.50pt}]
  \definecolor{color_r20g101b93}{rgb}{0.078,0.396,0.365}
  
  \coordinate (A) at (1.000, 1.000);
  \coordinate (B) at (5.000, 1.000);
  \coordinate (C) at (3.000, 4.464); 


  \begin{scope}
    \fill[ball_cover] 
      (3.000, 1.000) ellipse (2.3 and 0.7); 
  \end{scope}

  \begin{scope}
    \fill[ball_cover, rotate around={60:(2.000, 2.732)}] 
      (2.000, 2.732) ellipse (2.3 and 0.7);
  \end{scope}

  \begin{scope}
    \fill[ball_cover, rotate around={-60:(4.000, 2.732)}] 
      (4.000, 2.732) ellipse (2.3 and 0.7);
  \end{scope}

  \begin{scope}
    \fill[ball_cover] (A) circle (0.650);
  \end{scope}
  \begin{scope}
    \fill[ball_cover] (B) circle (0.650);
  \end{scope}
  \begin{scope}
    \fill[ball_cover] (C) circle (0.650);
  \end{scope}

  \node[node] at (A) {};
  \node[node] at (B) {};
  \node[node] at (C) {};

  \node[above] at (A) {$C_1$};
  \node[above] at (B) {$C_3$};
  \node[right] at (C) {$C_2$};
\end{tikzpicture}

}

\newsavebox{\QPanelQConerve}
\savebox{\QPanelQConerve}{
\begin{tikzpicture}[
 point/.style={circle, fill=black, inner sep=1.5pt},
 connection/.style={thick, opacity=0.8},
 scale=0.35
]
  \node[point] (A) at (0, 0) {}; 
  \node[point] (F) at (3, 0) {}; 
  \node[point] (H) at (1.5, 2.598) {}; 
   \foreach \label in {A, F, H} {
        \node[point] (\label) at (\label) {};
    }
    \node[above=3pt] at (H) {$C_2$};
    \node[left=3pt] at (A) {$C_1$};
    \node[right=3pt] at (F) {$C_3$};
    \draw[connection] (A) -- (F);
    \draw[connection] (F) -- (H);
    \draw[connection] (A) -- (H);
\end{tikzpicture}
}

\newsavebox{\QPanelConerve}
\savebox{\QPanelConerve}{
\begin{tikzpicture}[
  edge19/.style={black, opacity=0.8, thick},
  node/.style={circle, fill=black, inner sep=1.50pt},
  triangle/.style={fill=color_r230g57b71, opacity=0.3, draw=none},scale=0.6]
  \definecolor{color_r230g57b71}{rgb}{0.902,0.224,0.278}
  \fill[triangle] (1.000,3.000) -- (2.000,3.300) -- (1.500,2.200) -- cycle;
  \draw[edge19] (1.000,3.000) -- (2.000,3.300);
  \draw[edge19] (1.000,3.000) -- (1.500,2.200);
  \draw[edge19] (2.000,3.300) -- (1.500,2.200);
  \draw[edge19] (2.000,3.300) -- (3.000,3.700);
  \draw[edge19] (1.500,2.200) -- (2.500,1.000);
  \draw[edge19] (3.000,3.700) -- (3.800,3.000);
  \draw[edge19] (3.800,3.000) -- (3.200,1.800);
  \draw[edge19] (3.200,1.800) -- (2.500,1.000);
  \node[node] at (1.000,3.000) {};
  \node[node] at (2.000,3.300) {};
  \node[node] at (1.500,2.200) {};
  \node[node] at (3.000,3.700) {};
  \node[node] at (3.800,3.000) {};
  \node[node] at (3.200,1.800) {};
  \node[node] at (2.500,1.000) {};
\end{tikzpicture}
}



\begin{tikzpicture}
\node[anchor=center] {
\begin{tikzcd}[
    row sep=0.4cm,
    column sep=0.5cm,
    nodes={inner sep=2pt},
]
{\tikz[baseline=(current bounding box.center)] 
  \node {\usebox{\QPanelCov}};}
\arrow[r, maps to, "q"] \arrow[d, maps to, "\CoNerve"']
&
{\tikz[baseline=(current bounding box.center)] 
  \node {\usebox{\QPanelQ}};}
\arrow[d,maps to, "\CoNerve"]
\\
{\tikz[baseline=(current bounding box.center)] 
  \node {\usebox{\QPanelConerve}};}
\arrow[r,maps to,  "q"]
&
{\tikz[baseline=(current bounding box.center)] 
  \node {\usebox{\QPanelQConerve}};}
\end{tikzcd}
};
\end{tikzpicture}
    \caption{Quotienting a cover commutes with building its Co-Nerve. Given a cover and a partition $P = \{C_1, C_2, C_3\}$ of the space (top left), the Co-Nerve of the quotient cover (top left $\to$ top right $\to$ bottom right) equals the quotient of the Co-Nerve (top left $\to$ bottom left $\to$ bottom right).}
    \label{fig:quotient_complexes}
\end{figure}
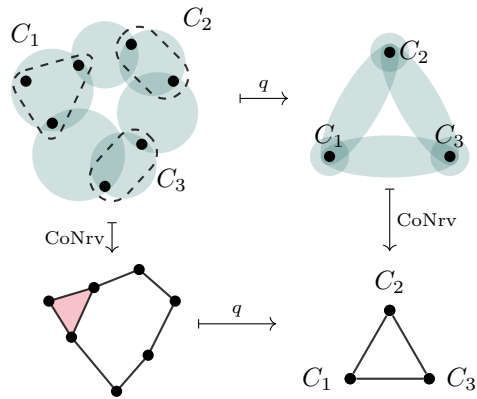

\item \textbf{New interleavings via persistent partitions.}
We introduce an opposing force to the growth of covers: a \textit{persistent partition}, a partition that progressively coarsens the space as the scale parameter increases.
By quotienting the ever-growing covers with this partition, we obtain a sequence of covers that remains small at every scale, while maintaining an interleaving (Fig.~\ref{fig:thick_simplicial_tower}).
Specifically, given a persistent partition $P$, we obtain the following approximations (multiplicative interleavings) between the ball cover ($\mathcal{B}$) and maximal-clique covers ($\mathcal{M}$):
$$
\mathcal{B}/P \xleftrightarrow{2} \mathcal{B}
\quad \text{and} \quad
\mathcal{M}/P \xleftrightarrow{3} \mathcal{M}.
$$
Since the underlying covers are interleaved any simplicial complex built from them will also have interleaved persistence modules.
Specifically (Fig.~\ref{fig:thick_simplicial_tower}) shows the Co-Nerve of $\mathcal{M}$ (Vietoris-Rips filtration) and  the Co-Nerve of $\mathcal{M}/P$ which are $\log 3$-interleaved because they are built from cover sequences $\mathcal{M}$ and $\mathcal{M}/P$ that are $\log 3$-interleaved.

\begin{figure}[htb]
    \centering
    \input{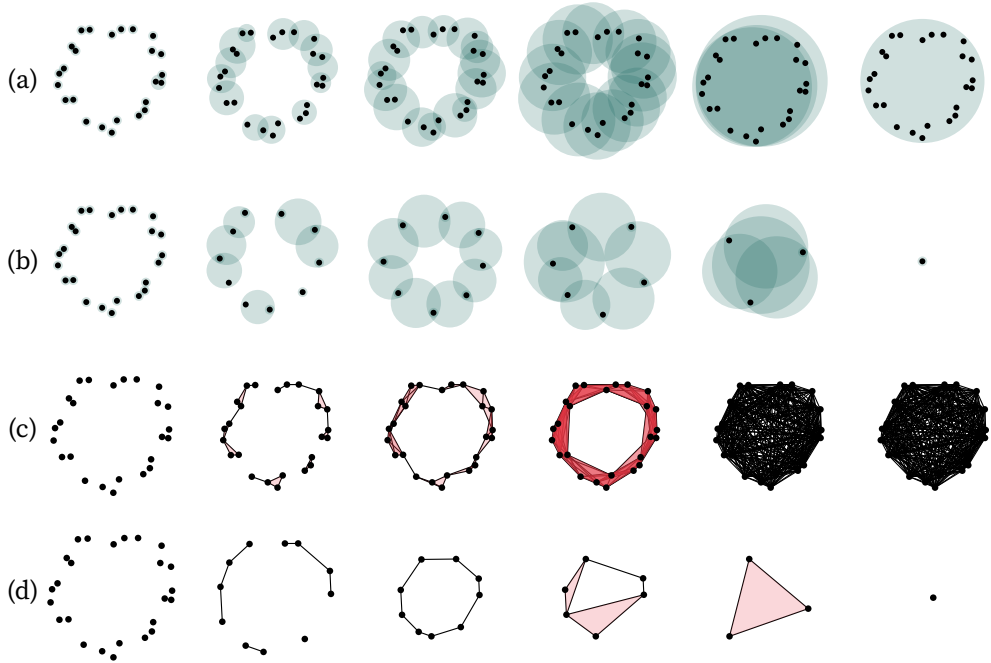}
    \caption{
    How approximations propagate from covers to simplicial complexes.
    The persistence modules of the Vietoris--Rips filtration \textbf{(c)} and our quotient filtration \textbf{(d)} are interleaved because they are the Co-Nerves of the cover sequences \textbf{(a)} and \textbf{(b)}, which are themselves interleaved.
    }
    \label{fig:thick_simplicial_tower}
\end{figure}

\item 
\textbf{Functorial hierarchical clustering.}
A valid persistent partition must satisfy two requirements: it must be \textit{functorial} and every block must have the diameter bounded by the current scale.
Hierarchical clustering is the natural source of sequences of partitions, and the question of which clustering methods are functorial was initiated by \citeauthor{carlsson10acharacterization}~\cite{carlsson10acharacterization, carlsson2010classifyingclusteringschemes}.
We show that complete linkage is the minimal linkage function that guarantees the diameter bound (Proposition~\ref{prop:complete_linkage_minimal}).
However, standard complete linkage fails functoriality: in the presence of ties, the output depends on the order in which pairs are
processed.
The functorial alternative of \citeauthor{carlsson10acharacterization}~\cite{carlsson10acharacterization} resolves this by merging transitively, but violates the diameter bound. We introduce \textit{conservative complete linkage} to restore functoriality while respecting the
diameter bound at every scale.

\item \textbf{Practical impact.}
Our construction produces a simplicial tower, which we convert to a filtration via the coning procedure of~\citeauthor{kerber2019barcodes}~\cite{kerber2019barcodes} and pass to Ripser~\cite{Bauer2021Ripser} for persistence computation; details are in Section~\ref{sec:computation}.
The resulting filtrations are orders of magnitude smaller than the standard Vietoris-Rips (Table~\ref{tab:filtration_comparison}), restoring near-linear scaling in the number of data points (Figure~\ref{fig:results_scalability}).
Our method reduces redundancy while efficiently preserving topology (Fig.~\ref{fig:results_complexity}): $H_2$ computations scale to $10^5$ points (Fig.~\ref{fig:benchmarks_torus}), and we detect the $H_5$ of a $5$-sphere (Figure~\ref{fig:benchmarks_h5}) from $2000$ points on a personal computer.

\end{enumerate}

\begin{table}[ht]
\centering
\caption{
Filtration size and persistence computation runtime across datasets of varying size $n$ and ambient dimension $d$.
Simplex counts up to dimension $2$ (M: millions, k: thousands) and runtimes compare the standard Vietoris-Rips filtration (VR) with our quotient construction.
All filtrations are computed up to the minimum enclosing radius $\min_{x \in X} \max_{y \in X} d(x,y)$, beyond which the Vietoris-Rips complex has trivial topology, and persistence is computed up to $H_1$.
}
\begin{tabular}{lccccccc}
\toprule
\multirow{2}{*}{Dataset} & \multirow{2}{*}{$n$} & \multirow{2}{*}{$d$} & \multicolumn{2}{c}{Filtration Size} & \multicolumn{2}{c}{Runtime (s)} \\
\cmidrule(lr){4-5} \cmidrule(lr){6-7}
 & & & \multicolumn{1}{c}{\small VR} & \multicolumn{1}{c}{\small Ours} & \multicolumn{1}{c}{\small VR} & \multicolumn{1}{c}{\small Ours} \\
\midrule
celegans & $297$ & $202$ & $4.1\,\mathrm{M}$ & $8.9\,\mathrm{k}$ & $0.019$ & $0.005$ \\
vicsek & $300$ & $3$ & $4.1\,\mathrm{M}$ & $2.2\,\mathrm{k}$ & $0.021$ & $0.004$ \\
klein\_400 & $400$ & $3$ & $6.7\,\mathrm{M}$ & $5.0\,\mathrm{k}$ & $0.092$ & $0.009$ \\
klein\_900 & $900$ & $3$ & $76.1\,\mathrm{M}$ & $15.4\,\mathrm{k}$ & $1.805$ & $0.050$ \\
dragon\_1k & $1000$ & $3$ & $56.4\,\mathrm{M}$ & $14.3\,\mathrm{k}$ & $0.239$ & $0.043$ \\
dragon\_2k & $2000$ & $3$ & $546.8\,\mathrm{M}$ & $30.4\,\mathrm{k}$ & $1.621$ & $0.182$ \\
HIV1 & $1088$ & $673$ & $157.1\,\mathrm{M}$ & $117.5\,\mathrm{k}$ & $0.563$ & $0.106$ \\
o3\_1024 & $1024$ & $9$ & $175.7\,\mathrm{M}$ & $118.6\,\mathrm{k}$ & $1.393$ & $0.097$ \\
o3\_2048 & $2048$ & $9$ & $1411.3\,\mathrm{M}$ & $305.7\,\mathrm{k}$ & $10.457$ & $0.373$ \\
pbmc3k & $2700$ & $50$ & $3261.3\,\mathrm{M}$ & $6.2\,\mathrm{M}$ & $3.660$ & $0.713$ \\
\bottomrule
\end{tabular}
\label{tab:filtration_comparison}
\end{table}

\begin{figure*}[!htb]
    \centering
    \begin{subfigure}[b]{0.32\textwidth}
        \centering
        \includegraphics[height=3.85cm]{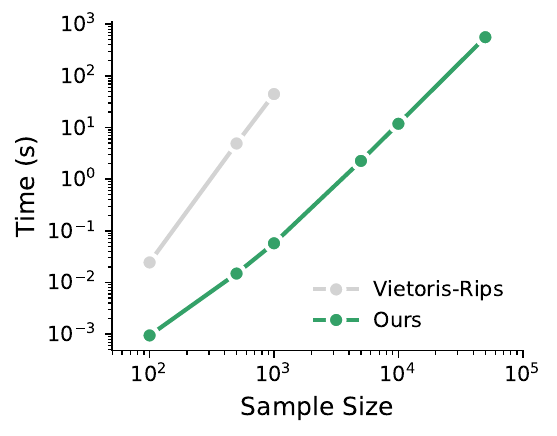}
        \caption{}
        \label{fig:benchmarks_torus}
    \end{subfigure}
    \hfill
    \begin{subfigure}[b]{0.32\textwidth}
        \centering
        \includegraphics[height=3.85cm]{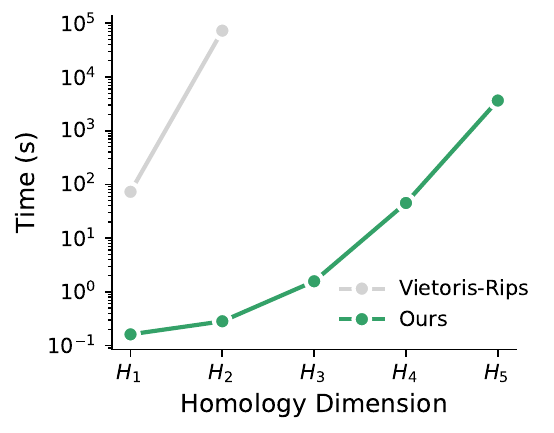}
        \caption{}
        \label{fig:benchmarks_sphere}
    \end{subfigure}
    \hfill
    \begin{subfigure}[b]{0.32\textwidth}
        \centering
        \resizebox{0.95\textwidth}{!}{%
           \input{content/figures/h5_persistence_diag.tex}
        }
        \caption{}
        \label{fig:benchmarks_h5}
    \end{subfigure}
    \caption{
    Runtime comparison between Vietoris-Rips filtration and our method.
    \textbf{(a)}~$H_2$ computation time for increasing sample sizes of a torus in $\mathbb{R}^3$.
    \textbf{(b)}~$H_n$ computation time for $2000$-point samples of $n$-spheres. 
    \textbf{(c)}~Persistence diagram from the $5$-sphere run, capturing the $H_5$ non-trivial class.
    }
    \label{fig:benchmarks}
\end{figure*}

\subsection{Related Work}
This work builds on several lines of research in topological data analysis.
While these directions are often presented independently, they are closely related through the structure of the underlying covers.

\paragraph{Covers and Dowker complexes.}
A foundational perspective in algebraic topology is the construction of simplicial complexes from covers, the Nerve being the classical example.
Dowker~\cite{dowker1952homology} generalized this by constructing two simplicial complexes from an arbitrary binary relation and its transpose, proving that they have isomorphic homology.
This result was later strengthened to a homotopy equivalence~\cite{bjorner1995topological}, shown to be natural with respect to inclusion of relations~\cite{chowdhury2018functorial, brun2019sparse} and cover refinements~\cite{virk2019rips}; for concise modern proofs of this duality, see \cite{yoon2026shortnewproofsdowker}.
The nerve construction is recovered as a special case when the binary relation is the inclusion relation of a cover (see \cite{desilva2025dowkerstheoremhigherorderrelations} for recent generalizations of Dowker's theorem to higher-order multiway relations).
Virk~\cite{virk2019rips} highlighted the Vietoris-Rips complex as one of the Dowker complexes associated with the maximal-clique cover, linking it to the Nerve of the same cover.

\paragraph{Sparse and reduced filtrations.}
The growth of simplicial filtrations is a fundamental bottleneck in persistent homology, and many strategies have been developed to mitigate it.
One direction focuses on selecting a sparser subset of the input: witness complexes~\cite{Silva2004TopologicalEU} select landmark points, while sparse Rips filtrations~\cite{sheehy2012linear,buchet2016efficient,buchet2024sparse} and generalized Dowker nerves~\cite{brun2019sparse,blaser2019sparsefilterednerves, blaser2019sparsenervespractice} construct topologically guaranteed approximations. 
Though theoretical linear-size bounds for these methods require a bounded doubling dimension, they still achieve significant practical compression on high-dimensional and non-metric data.
A complementary direction exploits Euclidean geometry to either constrain filtration growth, as in $\alpha$-complexes~\cite{edelsbrunner2003shape,graf2026floodcomplexlargescalepersistent}, or to accelerate the computation of Rips filtrations~\cite{koyama2024distilledvietorisripsfiltration,AGGARWAL2024102290}.
A third line of work reduces simplicial complexes via strong collapses~\cite{barmak2012strong}, which remove free faces without changing homotopy type.
In a filtered setting, collapses can preserve persistent homology~\cite{dlotko2013simplification,boissonnat2019computing,glisse2022swap}, or be used more aggressively to obtain efficient approximations of Vietoris-Rips filtrations~\cite{dey2019simba,dey2014computing}. 
Together, these approaches achieve significant sparsification, but typically under assumptions on geometry, ambient dimension, or homological degree.

\paragraph{Simplicial Towers}.
There is ongoing work on computing persistent homology for simplicial towers, which generalize filtrations by allowing vertex contractions besides inclusions.
\citeauthor{dey2014computing}~\cite{dey2014computing} showed that a single vertex contraction can be replaced by a small number of simplex inclusions without altering persistence, and \citeauthor{kerber2019barcodes}~\cite{kerber2019barcodes} extended this to full towers, converting any simplicial tower into a filtration of only marginally larger size.
We rely on their construction to compute the persistence of our quotient filtrations.

\paragraph{Relaxed Interleavings.}
The standard notion of interleaving requires strict commutativity of diagrams.
\citeauthor{blumberg2022universalityhomotopyinterleavingdistance}~\cite{blumberg2022universalityhomotopyinterleavingdistance} introduced homotopy interleavings, where diagrams commute only up to homotopy, and proved that a homotopy interleaving between filtrations induces a strict interleaving between persistence modules.
\citeauthor{desilva2018theoryinterleavingscategoriesflow}~\cite{desilva2018theoryinterleavingscategoriesflow} formalized this in a categorical setting: if a functor from a source category to a target category preserves the relevant structure, interleavings in the source induce interleavings in the target.
Our notion of \textit{approximation up to contiguity} (Definition~\ref{def:c-approximation-contiguity}) is the cover-level analogue of this idea: contiguity of refinement maps plays the role of homotopy, and the passage to homology converts approximate commutativity into strict equality.

\section{Background}

\subsection{Category Theory}

We begin by establishing the categorical framework that allows us to propagate approximations through functorial constructions.
A \textit{functor} $F : \Cat{C} \to \Cat{D}$ between two categories assigns to each object $A \in \Cat{C}$ an object $F(A) \in \Cat{D}$ and to each morphism $f : A \to B$ a morphism $F(f) : F(A) \to F(B)$, preserving identity and composition: $F(\mathrm{id}_A) = \mathrm{id}_{F(A)}$ and $F(g \circ f) = F(g) \circ F(f)$.

A \textit{natural transformation} $\eta : F \Rightarrow G$ between two functors $F, G : \Cat{C} \to \Cat{D}$ assigns to each object $A \in \Cat{C}$ a morphism $\eta_A : F(A) \to G(A)$ in $\Cat{D}$ such that for every morphism $f : A \to B$ the following diagram commutes:
\[
\begin{tikzcd}
F(A) \ar[r, "\eta_A"] \ar[d, "F(f)"'] & G(A) \ar[d, "G(f)"] \\
F(B) \ar[r, "\eta_B"'] & G(B)
\end{tikzcd}
\]
We say that the diagram commutes if 
$$
 G(f) \circ \eta_A = \eta_B\circ F(f)
$$
A natural transformation whose components are all isomorphisms is a \textit{natural isomorphism}.

A \textit{persistent object} is a functor $F: (\mathbb{R}_{+}, \leq) \to \Cat{C}$, for some category $\Cat{C}$ where $(\mathbb{R}_{+}, \leq)$ is viewed as a poset category.
A persistent object is a rigorous description of how an object of a certain category changes across some parameter.
There is a specific type of persistent object called a \textit{persistence module} which is central to persistent homology:

\begin{definition}[Persistence Module]
A \textbf{persistence module} is a functor $M: (\mathbb{R}_+,\leq)\to \Cat{Vect}$, assigning to each $r$ a vector space $M(r)$ and to each $r\leq s$ a linear map $M(r\leq s):M(r) \to M(s)$.
\end{definition}

The key notion of comparison between persistent objects is the interleaving, which quantifies how far two functors are from being isomorphic.

\begin{definition}[Interleaving  \cite{chazal2009proximity,lesnick2015theory,oudot2015persistence}]
Let $F,G:\poset\to \Cat{C}$ be persistent objects.
An $\varepsilon$-interleaving between $F$ and $G$ is a pair of natural transformations
\[
\phi: F\Longrightarrow G\circ T_\varepsilon,\qquad \psi: G\Longrightarrow F\circ T_\varepsilon,
\]
where $T_\varepsilon(t) = t+\varepsilon$, such that for every $t\in\mathbb R_+$ the following coherence relations hold:
\begin{align}
\label{eq:add-coh-1}
\psi_{t+\varepsilon}\circ \phi_{t} &\;=\; F(t\le t+2\varepsilon),\\
\label{eq:add-coh-2}
\phi_{t+\varepsilon}\circ \psi_{t} &\;=\; G(t\le t+2\varepsilon).
\end{align} represented by the coherence diagrams:
    
    \[
    \begin{tikzcd}[column sep=tiny]
        F(t) \arrow[rr] \arrow[dr, "\phi_t"'] & & F(t+2\varepsilon) \\
        & G(t+\varepsilon) \arrow[ur, "\psi_{t+\varepsilon}"'] &
    \end{tikzcd}
    \qquad \qquad
    \begin{tikzcd}[column sep=tiny]
        G(t) \arrow[rr] \arrow[dr, "\psi_t"'] & & G(t+2\varepsilon) \\
        & F(t+\varepsilon) \arrow[ur, "\phi_{t+\varepsilon}"'] &
    \end{tikzcd}
    \]
\end{definition}

Sometimes instead of a translation we have a scale, where the shift factor between the functors is not additive but multiplicative, which is most common in TDA literature \cite{lesnick2015theory,desilva2018theoryinterleavingscategoriesflow,kerber2013approximate}.
The multiplicative analogue of the interleaving is an \textit{approximation}.
\begin{definition}[$c$-Approximation]
\label{def:c-approximation}
Let $F, G: \poset \to \Cat{C}$ be two persistence objects.
We say that $F$ and $G$ are \textbf{$c$-approximated} if there exist natural transformations $\phi: F \Rightarrow G \circ S_c$ and $\psi: G \Rightarrow F \circ S_c$ 
(where $S_c(r) = cr$) such that:
\begin{align}
\psi_{cr} \circ \phi_r  &\;=\; F(r\leq c^2r),\\
\phi_{cr} \circ \psi_r &\;=\; G(r \leq c^2r).
\end{align}
equivalently we say that the following diagrams commute:

\begin{equation}
\begin{tikzcd}[column sep=tiny]
F(r) \arrow[rr] \arrow[dr, "\phi_r"'] & & F(c^2r) \\
& G(cr) \arrow[ur, "\psi_{cr}"'] &
\end{tikzcd}
\quad \text{and} \quad
\begin{tikzcd}[column sep=tiny]
G(r) \arrow[rr] \arrow[dr, "\psi_r"'] & & G(c^2r) \\
& F(cr) \arrow[ur, "\phi_{cr}"'] &
\end{tikzcd}
\end{equation}
We say that $F$ and $G$ are $c-$approximated and denote as:
$$\phi:F \xleftrightarrow{c} G:\psi$$
\end{definition}

The two notions are related by a logarithmic reparameterization.
An approximation can be converted into an interleaving by a reparameterization to the logarithmic scale.
That is if two functors are $c$-approximated then they are $\log(c)$-interleaved \cite{kerber2013approximate}.
\begin{proposition}
\label{prop:mult-to-add-categorical}
Let $F,G:\poset\to \Cat{C}$ be persistent objects and suppose there exists a $c$-approximation
$(\varphi,\psi)$ between them. Define the reindexed functors
\[
\widetilde F,\widetilde G : \poset \longrightarrow\Cat{C},\qquad
\widetilde F(s):=F(e^s),\quad \widetilde G(s):=G(e^s),
\]
and for $s\le t$ set $\widetilde F(s\le t):=F(e^s\le e^{t})$ (and similarly for $\widetilde G$). Then
\[
\widetilde\varphi_s:=\varphi_{e^s},\qquad \widetilde\psi_s:=\psi_{e^s}
\]
define natural transformations
\[
\widetilde\varphi:\widetilde F\Longrightarrow \widetilde G\circ T_{\ln c},\qquad
\widetilde\psi:\widetilde G\Longrightarrow \widetilde F\circ T_{\ln c},
\]
and the coherence relations \eqref{eq:add-coh-1}--\eqref{eq:add-coh-2} hold with $\varepsilon=\log c$.
In other words, $(\widetilde\varphi,\widetilde\psi)$ is a $\log c$-interleaving between $\widetilde F$ and $\widetilde G$.
\end{proposition}

The existence of an $\varepsilon$-interleaving provides a quantitative measure of similarity between two persistent objects, which leads to the definition of the interleaving distance.

\begin{definition}[Interleaving Distance]
The \textbf{interleaving distance} $d_I(F, G)$ between two persistent objects $F$ and $G$ is the infimum of all $\varepsilon \ge 0$ for which they are $\varepsilon$-interleaved:
\[
d_I(F,G) := \inf\{\varepsilon \ge 0 \mid F \text{ and } G \text{ are } \varepsilon\text{-interleaved}\}
\]
If no such $\varepsilon$ exists, the distance is defined to be $\infty$.
\end{definition}

Finally, we record that approximations are stable under natural isomorphisms.
This is implicit in the standard definitions, where the interleaving distance is defined on isomorphism classes of persistence modules~\cite{chazal2009proximity,lesnick2015theory,bubenik2015metrics} (i.e. naturally isomorphic functors have interleaving distance $0$).
We include a short proof since we need the explicit approximation maps when passing between nerves, co-nerves, and their pullbacks via Dowker duality.

\begin{lemma}
\label{lem:iso-transport}
Let $F, F', G: \poset \to \Cat{C}$ be persistent objects with 
$\phi: F \xleftrightarrow{c} G : \psi$.
If $\theta: F \Rightarrow F'$ is a natural isomorphism, then
$$F' \xleftrightarrow{c} G$$
via the composition maps $\phi' := \phi \circ \theta^{-1}$ and $\psi' := \theta \circ \psi$.
\end{lemma}

\begin{proof}
Define maps $\phi'$ and $\psi'$ by composing with the isomorphism $\theta$.
\begin{equation*}
    \begin{tikzcd}[ampersand replacement=\&]
        F(r) \arrow[r, "\phi_r"] \arrow[d, "\theta_r"'] \& G(cr) \\
        F'(r) \arrow[ur, dashed, "\phi'_r"']
    \end{tikzcd}
    \qquad
    \phi'_r := \phi_r \circ \theta_r^{-1}
    \qquad \text{and} \qquad
    \psi'_r := \theta_{cr} \circ \psi_r
\end{equation*}

\textbf{Naturality of $\phi'$:} For $r \leq s$, the commutativity of the outer rectangle follows from the commutativity of the inner squares (naturality of $\theta^{-1}$ and $\phi$):
\[
\begin{tikzcd}[row sep=large, column sep=huge, ampersand replacement=\&]
    F'(r) \arrow[r, "\theta_r^{-1}"] \arrow[d, "F'(r \leq s)"'] 
    \& F(r) \arrow[r, "\phi_r"] \arrow[d, "F(r \leq s)"] 
    \& G(cr) \arrow[d, "G(cr \leq cs)"] \\
    F'(s) \arrow[r, "\theta_s^{-1}"'] 
    \& F(s) \arrow[r, "\phi_s"'] 
    \& G(cs)
\end{tikzcd}
\]
Algebraically:
\begin{align*}
\phi'_s \circ F'(r \leq s) 
&= \phi_s \circ \theta_s^{-1} \circ F'(r \leq s) \\
&= \phi_s \circ F(r \leq s) \circ \theta_r^{-1} 
\quad \text{($\theta^{-1}$ natural)} \\
&= G(cr \leq cs) \circ \phi_r \circ \theta_r^{-1} 
\quad \text{($\phi$ natural)} \\
&= G(cr \leq cs) \circ \phi'_r
\end{align*}
Naturality of $\psi'$ is similar.

\textbf{Coherence:} We show $\psi'_{cr} \circ \phi'_r = F'(r \leq c^2r)$. The diagram below traces the path from $F'(r)$ to $F'(c^2r)$. The central step reduces via the interleaving of $F$ and $G$, and the outer shell holds by the naturality of $\theta$.

\[
\begin{tikzcd}[column sep=large, row sep=large, ampersand replacement=\&]
    F'(r) \arrow[r, "\theta_r^{-1}"] \arrow[d, dashed, "F'(r \leq c^2r)"']
    \& F(r) \arrow[r, "\phi_r"] \arrow[d, "F(r \leq c^2r)" description]
    \& G(cr) \arrow[d, "\psi_{cr}"] \\
    F'(c^2r) 
    \& F(c^2r) \arrow[l, "\theta_{c^2r}"] 
    \& F(c^2r) \arrow[l, equal]
\end{tikzcd}
\]
\begin{align*}
\psi'_{cr} \circ \phi'_r 
&= (\theta_{c^2r} \circ \psi_{cr}) \circ (\phi_r \circ \theta_r^{-1}) \\
&= \theta_{c^2r} \circ (\psi_{cr} \circ \phi_r) \circ \theta_r^{-1} \\
&= \theta_{c^2r} \circ F(r \leq c^2r) \circ \theta_r^{-1} \quad \text{($\phi, \psi$ approximate)} \\
&= F'(r \leq c^2r) \quad \text{($\theta$ natural)}
\end{align*}
The second condition $\phi'_{cr} \circ \psi'_r = G(r \leq c^2r)$ is verified similarly.
\end{proof}

\subsection{Simplicial complexes and simplicial maps}
Given a vertex set $V$, a $k$-simplex is a subset of $k+1$ elements of $V$.
An (abstract) simplicial complex $K$ built from the vertex set $V$ is a set containing simplices of $V$ and closed under taking subsets.
More precisely if a $k$-simplex $\sigma$ is in $K$ then so must all its subsets.
A map $\phi: K\to L$ is called a \textit{simplicial map} if the image of every simplex $\sigma=\{v_1\dots,v_{k+1} \} \in K$, given by $\phi(\sigma)=\{\phi(v_1), \dots ,\phi(v_{k+1})\}$, is a simplex in $L$.
By definition, a simplicial map is completely determined by its effect on the vertices.

In the simplicial setting contiguity is the discrete analogue of homotopy.
Contiguous simplicial maps encode a notion of closeness, where $f(\sigma)$ and $g(\sigma)$ are both faces of a larger simplex $\tau$.
\begin{definition}[Contiguous Simplicial Maps~\mbox{\parencite[p. 62]{munkres2018elements}}]
    Given two simplicial maps $f,g:K\to L$, these maps are said to be \textbf{contiguous} if for each simplex $\sigma\in K$ the simplex $f(\sigma)\cup g(\sigma)$ forms a simplex in $L$
\end{definition}
If two simplicial maps are contiguous, they are chain homotopic~\cite[Theorem 12.5]{munkres2018elements} and induce the same homomorphism on homology~\cite[Theorem 12.4]{munkres2018elements}. 
We will establish (Section~\ref{sec:approximation_of_covers}) an analogous notion at the cover level, contiguity of refinement maps, and show that it implies contiguity of the induced simplicial maps.

\subsection{Dowker Complexes}

We now introduce the central objects of this work.
Dowker complexes provide a general mechanism for constructing simplicial complexes from binary relations.



\begin{definition}[Dowker Complexes]
Given a binary relation $R \subseteq X \times Y$, we define a simplicial complex $C(R)$ whose vertex set is $X$. A finite subset $\sigma \subseteq X$ forms a simplex in $C(R)$ if its elements share a common witness in $Y$:
$$C(R) = \{\sigma \subseteq X \mid \exists y \in Y \text{ such that } \sigma \times \{y\} \subseteq R\}.$$

From the same relation, we can create a second simplicial complex by considering the transpose relation $R^{op} \subseteq Y \times X$, where $(y,x) \in R^{op}$ if and only if $(x,y) \in R$. Applying the same construction to the transpose relation yields a complex $C(R^{op})$ whose vertex set is $Y$:
$$C(R^{op}) = \{\tau \subseteq Y \mid \exists x \in X \text{ such that } \tau \times \{x\} \subseteq R^{op}\}.$$

The two complexes built from $R$ and $R^{op}$ are fundamentally linked and are together called the Dowker Complexes. 
\end{definition}
Dowker's theorem states that $C(R)$ and $C(R^{op})$ have isomorphic homology groups.

\begin{theorem}[Dowker Duality~\cite{dowker1952homology}]
\label{thm:dowker}
Let $C(R)$ and $C(R^{op})$ be two simplicial complexes built as above, for each $k \in \mathbb{Z}_+$ we have isomorphisms:
$$ H_k C(R) \cong H_k C(R^{op})$$
\end{theorem}

A stronger form of the Dowker Theorem was proved by \citeauthor{bjorner1995topological}~\cite{bjorner1995topological} showing that the geometric realizations of the Dowker complexes are homotopy equivalent:  $\lvert C(R) \rvert \simeq \lvert C(R^{op})\rvert$.
This was then generalized as a specific case of the \textit{Functorial Dowker Theorem}~\cite{chowdhury2018functorial}.
Given two nested binary relations $S\subseteq R\subseteq X \times Y$ the Functorial Dowker Theorem states the there exist homotopy equivalences $\Gamma_S:\vert C(S)\rvert \to \lvert C(S^{op})\rvert$ and $\Gamma_R:\vert C(R)\rvert \to \lvert C(R^{op})\rvert$ such that following diagrams commutes up to homotopy:
\[
\begin{tikzcd}
\lvert C(S) \rvert \ar[r,hookrightarrow] \ar[d,"\Gamma_S"']
&
\lvert C(R)\rvert \ar[d,"\Gamma_{R}"]
\\
\lvert C(S^{op})\rvert \ar[r,hookrightarrow]
&
\lvert C(R^{op})\rvert
\end{tikzcd}
\]
\citeauthor{brun2019sparse}~\cite{brun2019sparse} organized this inclusion-based naturality into a 2-category of Dowker dissimilarities and show that these relation inclusions induce contiguous simplicial maps under the nerve~\cite[Lemma 6.4]{brun2019sparse}.

\subsection{Covers and Refinements}

\begin{definition}[Cover]
Given a set $X$, a \textbf{cover} of $X$ is a family of sets $\mathcal{U}=\{U_i\}_{i\in I}$ indexed over some set $I$ such that $X = \bigcup_{i\in I}U_i$.
\end{definition}

Throughout the work we consider $X$ as a metric space, with no extra structure.

\begin{definition}[Refinement]
Let $\mathcal{U} = \{U_i\}_{i \in I}$ and $\mathcal{V} = \{V_j\}_{j \in J}$ be two covers of a space $X$.  
A \textbf{refinement map} is a function $f : J \to I$ such that 
\[
V_j \subseteq U_{f(j)} \quad \text{for all } j \in J.
\]
If such a map exists, we say that $\mathcal{V}$ \emph{refines} $\mathcal{U}$ (or is a \emph{refinement} of $\mathcal{U}$).
\end{definition}

For notational convenience, we will often write $f : \mathcal{V} \to \mathcal{U}$ and interpret this as the refinement map induced by $f : J \to I$, so that $f(V_j) := U_{f(j)}$. 
In this way, the notation emphasizes that the map acts between covers, even though it formally acts on their index sets.

The set of covers of $X$ and the set of refinement maps between them define a category.

\begin{definition}[Category of Covers]
Let $\Cat{Cov}(X)$ be the category whose objects are covers of $X$ and the maps are refinement maps.
\end{definition}

\section{Approximation of Covers}
\label{sec:approximation_of_covers}
We now present the path to propagate approximations from covers to persistence modules.
The strategy rests on the fact that all refinement maps between two covers are contiguous, meaning that all diagrams of covers commute up to contiguity (Proposition~\ref{prop:all_refinements_contiguous}).
We then show that contiguity at the cover level becomes strict equality at the homology level (Theorem~\ref{thm:main_propagation}).

\subsection{The Nerve and Co-Nerve}
We start with a cover $\mathcal{U} = \{U_i\}_{i\in I} \in \Cat{Cov}(X)$ we associate to it a binary relation encoding set membership.
\begin{definition}[Inclusion Relation]
Given a cover $\mathcal{U} = \{U_i\}_{i\in I}$ of a space $X$, indexed by a set $I$ we consider the binary relation $R_\mathcal U:X\times I$ where: $$(x,i) \in R_\mathcal U\iff x \in U_i$$
\end{definition}
This relation is the bridge between covers and simplicial complexes: the two Dowker complexes of $R_{\mathcal{U}}$ yield two different simplicial complexes from the same cover.
The $X$-side complex consists of subsets of the space that are contained in a common cover element; the $I$-side complex consists of collections of cover elements with non-empty common intersection.
We will call these complexes the Co-Nerve and the Nerve

\begin{definition}[Co-Nerve]
The \textbf{co-nerve} of a cover $\mathcal{U}=\{U_i\}_{i\in I}$ is the simplicial complex:
\begin{align*}
    \CoNerve(\mathcal{U}) &= \{\sigma \subseteq X \mid \exists i\in I \quad \sigma \times \{i\} \subset R\}\\
        &=\{\sigma \subseteq X \mid \exists U_i\in\mathcal{U} \quad \sigma\subseteq U_i\}\\
        &=C(R_\mathcal U)\\
\end{align*}
\end{definition}

\begin{definition}[Nerve]
The \textbf{nerve} of a cover $\mathcal{U}=\{U_i\}_{i\in I}$ is the simplicial complex:
\begin{align*}
    \Nerve(\mathcal{U})&= \{\sigma \subseteq I \mid \exists x\in X \quad \{x\} \times \sigma \subset R\}\\
               &= \{\sigma \subseteq I  \mid \bigcap_{i\in \sigma} U_i \neq \emptyset \}\\
               &= C(R^{op}_\mathcal U)
\end{align*}
\end{definition}

The Dowker Duality~\ref{thm:dowker} guarantees an isomorphism at the homology level:
$$
H_*\Nerve(\mathcal{U})  \cong H_* \CoNerve(\mathcal{U})
$$ 
\begin{remark}[Co-nerve naming and notation]
While some authors refer to the first construction (co-nerve) as the \textit{Vietoris complex of a cover}~\cite{begle1950bicompact,hatcher2002algebraic,virk2019rips} or the \textit{Vietoris Nerve}~\cite{edwards1980cech}
we prefer the term \textit{Co-nerve}. 
The Co-Nerve, like the Nerve, is best viewed as a functor from covers to simplicial complexes.
The Vietoris-Rips complex is the Co-Nerve of a \textit{specific cover} just as the \v Cech complex is the Nerve of the ball cover~\cite{edelsbrunner2010computational,chazal2009gromov,chazal2021introduction,cavanna2015geometric}.
Not every Nerve is a \v Cech Complex and not every Co-Nerve results in the Vietoris-Rips complex. In fact, as we will show, the properties of both simplicial complexes come from their underlying covers.

The Dowker Theorem provides an initial argument to this, the fact that two different simplicial complexes built on the same underlying cover are homotopy equivalent is an early indication that the topological information is determined with their common element: the cover.
Therefore the naming helps to distinguish the role of the simplicial constructions from the role of the covers.
\end{remark}

\subsection{Persistent Covers}

As is common in persistent homology we are interested in sequences of covers of $X$ indexed by some parameter, with the constraint that each cover refines the next one.

\begin{definition}[Persistent Cover]
A \textbf{persistent cover} is a functor ${\mathcal{U}: (\mathbb{R}_{+}, \leq) \to \Cat{Cov}(X)}$. For each $r \in \mathbb{R}_{+}$, we have a cover $\mathcal{U}(r)$ of $X$, and for $r \leq s$, we have a refinement map $\mathcal{U}(r \leq s): \mathcal{U}(r) \to \mathcal{U}(s)$.
\end{definition}

A persistent cover is a functor $\mathcal{U} :\poset \to \Cat{Cov}(X)$: at each scale we have a cover, and as the scale increases the covers are related by refinement maps.
Since the objects $\mathcal{U}\in \Cat{Cov}(X)$ are subsets of $X$, this functor is generally known as a persistent set~\cite{carlsson10acharacterization,cardona2022universalellpmetricmergetrees}.
However we want to emphasize that they have extra constraints (the union must be equal to $X$).

\citeauthor{virk2019rips}~\cite{virk2019rips} further expands the \textit{Functorial Dowker Theorem}~\cite{chowdhury2018functorial} in a different direction, showing that the Dowker isomorphisms are natural with respect to cover refinements.
A refinement map is not in general an inclusion of relations, so naturality at this level extends beyond inclusion-based results~\cite{chowdhury2018functorial,brun2019sparse}.
Given a refinement $f:\mathcal{V}\to \mathcal U$, there exists homotopy equivalences $\gamma_U$ and $\gamma_V$ for which the following diagram commutes up to homotopy:
$$
\begin{tikzcd}
\CoNerve(\mathcal{V}) \arrow[r,hook, "\CoNerve(f)"'] \arrow[d, "\simeq","\gamma_V"'] & 
\CoNerve(\mathcal{U}) \arrow[d, "\simeq"', "\gamma_U"] \\
\Nerve(\mathcal{V})\arrow[r, "\Nerve(f)"] & 
\Nerve(\mathcal{U})
\end{tikzcd}
$$
Which means that given a persistent cover $\mathcal{U}$, the Dowker isomorphisms
$$\theta_r: H_*\Nerve\;\mathcal{U}(r) \xrightarrow{\cong} H_*\CoNerve\;\mathcal{U}(r)$$
are natural to the persistence structure.
That is, for $r\leq s$, the following diagram commutes:
$$
\begin{tikzcd}
H_*\CoNerve\;\mathcal{U}(r) \arrow[r] \arrow[d, "\theta_r"'] & 
H_*\CoNerve\;\mathcal{U}(s) \arrow[d, "\theta_s"'] \\
H_*\Nerve\;\mathcal{U}(r)\arrow[r] & 
H_*\Nerve\;\mathcal{U}(s) 
\end{tikzcd}
$$

\subsection{Contiguity of cover refinements}
Having defined the Nerve and Co-Nerve as functors out of $\Cat{Cov}(X)$, we now examine how they act on refinement maps.

\begin{lemma}
\label{lem:refinement-inclusion}
Let $f:\mathcal{V} \to \mathcal{U}$ be a refinement map.
\begin{enumerate}
  \item The map on nerves induced by $f$,
  \[
    \Nerve(f)\colon \Nerve(\mathcal V)\longrightarrow \Nerve(\mathcal U),
  \] is an inclusion of simplicial complexes \emph{if and only if} $f$ is injective.
  \item  The map on co-nerves induced by $f$,
  \[
    \CoNerve(f)\colon \CoNerve(\mathcal V)\longrightarrow \CoNerve(\mathcal U)
  \]
  is always an inclusion $\CoNerve(\mathcal V)\subseteq\CoNerve(\mathcal U)$.
\end{enumerate}
\end{lemma}

\begin{proof}
(1)
A simplicial map is injective if and only if it is injective on its set of vertices. 
By the construction of the Nerve functor, the vertex set of $\Nerve(\mathcal{V})$ is the index set $I$, and the action of $\Nerve(f)$ on this vertex set is given precisely by the function $f: \mathcal{V} \to \mathcal{U}$. 
Thus, $\Nerve(f)$ is an inclusion map if and only if the function $f$ is injective.

(2)
Consider a simplex $\sigma \in \CoNerve(\mathcal{V})$.
By definition there exists some $V \in \mathcal{V}$ such that $\sigma \subseteq V$.
Since $\mathcal{V}$ refines $\mathcal{U}$, there is some set $f(V)\in \mathcal{U}$ such that $V\subseteq f(V)$.
Which implies that $\sigma \subseteq V \subseteq f(V)$.
The existence of a set in $\mathcal{U}$ containing $\sigma$ means that $\sigma \in \CoNerve(\mathcal{U})$.
\end{proof}
The asymmetry in this lemma has a direct consequence: the Co-Nerve will always produce a nested filtration while the Nerve may not.

Given a refinement map $f:\mathcal{V}\to \mathcal{U}$, often a set $V_j$ is contained in \textit{multiple} sets in $\mathcal U$. 
Which makes the refinement map not unique and one has to make a choice (Fig. \ref{fig:contiguous_maps}).
However these different possible refinement maps are ``close'' in some sense and become identical linear maps at the homology level: they are contiguous.

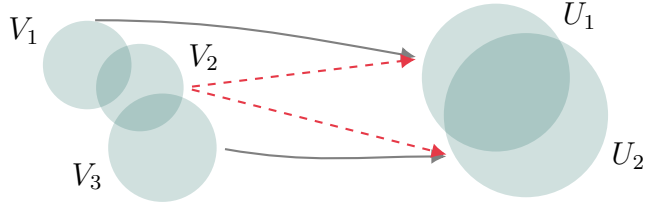
\begin{figure}[htb]
    \centering
    \begin{tikzpicture}[
  blob/.style={
    circle, 
    fill=teal, 
    opacity=0.2, 
    inner sep=0pt, 
    minimum width=#1*2cm
  },
  set_label/.style={
    font=\sffamily\large,
    text=black,
    inner sep=2pt
  },
  map_arrow/.style={
    ->, 
    >={LaTeX[width=2mm,length=2mm]}, 
    draw=gray!100, 
    thick, 
    shorten >= 3pt, 
    shorten <= 3pt
  },
  split_arrow/.style={
    map_arrow,
    draw=redd, 
    dashed
  }
]

  \definecolor{teal}{rgb}{0.078,0.396,0.365}
  \definecolor{redd}{rgb}{0.902,0.224,0.278}

  \begin{scope}[xshift=2cm, local bounding box=LeftSet]
    \node[blob=0.589, label={[set_label]160:$V_1$}] (s1) at (2.500,3.500) {};
    
    \node[blob=0.582, label={[set_label]15:$V_2$}] (s2) at (3.200,3.200) {};
    
    \node[blob=0.721, label={[set_label]190:$V_3$}] (s3) at (3.500,2.400) {};
  \end{scope}

  \begin{scope}[xshift=7cm, local bounding box=RightSet]
    \node[blob=0.984, label={[set_label]35:$U_1$}] (b1) at (2.933,3.333) {};
    
    \node[blob=1.092, label={[set_label]345:$U_2$}] (b2) at (3.333,2.833) {};
  \end{scope}

  
  \draw[map_arrow] (s1.north) to[out=0, in=170] (b1.165);

  \draw[map_arrow] (s3.east) to[out=350, in=180] (b2.210);

  \draw[split_arrow] (s2.east) to (b1.165);
  \draw[split_arrow] (s2.east) to (b2.210);

\end{tikzpicture}
    \caption{Example of two contiguous refinement maps $f,g:\mathcal V \to \mathcal U$ defined as $f(V_1)=f(V_2)=U_1$, $f(V_3)=U_2$ and $g(V_1)=U_1$, $g(V_2)=g(V_3)=U_2$}
    \label{fig:contiguous_maps}
\end{figure}

\begin{definition}[Contiguous Refinement Maps]
Let $\mathcal{U} = \{U_i\}_{i \in I}$ and $\mathcal{V} = \{V_j\}_{j \in J}$ be two covers of a space $X$ with $\mathcal{V}<\mathcal{U}$.
Two refinement maps $f, g: J \to I$ are said to be \textbf{contiguous} if for every $j \in J$, the images intersect:
\[
U_{f(j)} \cap U_{g(j)} \neq \emptyset
\]
\end{definition}
Intuitively this gives a notion of closeness to two refinements maps $f$ and $g$, which may assign different indices but correspond to cover elements that intersect.
Simplicial maps induced by contiguous refinement maps are also going to be contiguous.

\begin{lemma}[Homology Invariance]
\label{lem:homology_invariance_contiguity}
    Let $f,g:\mathcal V\to\mathcal{U}$ be contiguous refinement maps.
    Then they induce the same linear maps on homology for both the Nerve and the Co-Nerve:
    \begin{align*}
        H(\Nerve(f))&=H(\Nerve(g)) \tag{1}\\
         H(\CoNerve(f))&=H(\CoNerve(g)) \tag{2}\\
    \end{align*}
\end{lemma}
\begin{proof}
(1). \textit{Nerve:} We will show that contiguous refinement maps induce contiguous simplicial maps\footnote{A proof of this step can also be found in \cite[Lemma 2.1]{dey2016multiscale}. In the context of the Mapper algorithm the authors prove that any two refinement maps between covers induce contiguous simplicial maps between their nerves.}.
Let $\sigma=\{V_1, \dots, V_k\} \in \Nerve(\mathcal{V})$.
By definition, this means that $\cap_{j=1}^k V_j\neq \emptyset$.
Let $x \in \cap_{j=1}^k V_j$  be an element in the intersection.

Since $f$ and $g$ are refinement maps, we have the inclusions $V_j \subseteq U_{f(j)}$ and $V_j \subseteq U_{g(j)}$ for all $j=1,\dots,k$.
Consequently, the point $x$ is contained in every set in the image of the simplex under both maps:
\[
x \in \left( \bigcap_{j=1}^k U_{f(j)} \right) \cap \left( \bigcap_{j=1}^k U_{g(j)} \right).
\]
Thus, the set of vertices $f(\sigma) \cup g(\sigma)$  has a non-empty common intersection (containing at least $x$).
By the definition of the Nerve, this means $f(\sigma) \cup g(\sigma)$ forms a simplex in $\Nerve(\mathcal{U})$.
The induced simplicial maps $\Nerve(f)$ and $\Nerve(g)$ are contiguous and therefore they induce the same homomorphism on homology.

(2). \textit{Co-Nerve:} A refinement map $f: \mathcal{V} \to \mathcal{U}$ acts on the index sets of the covers, but induces the identity map (or inclusion) on the underlying vertex set $X$.
Thus both induced maps $\CoNerve(f)$ and $\CoNerve(g)$ act as the identity on the vertices.
Therefore, as simplicial maps, $\CoNerve(f) = \CoNerve(g)$, which trivially implies they induce the same linear maps on homology.
\end{proof}

Contiguous maps help us relax the notion of commutativity at the cover level without compromising it at the homology level.
Given three refinement maps $f,g,h$ it is often easier to check contiguity ($f\circ g\sim h$) than equality ($f\circ g=h$).
While this allows us to side-step a strict equality check, we \emph{still} need to confirm contiguity.
However one can show that all refinement maps between two given covers are contiguous.
Which means that we don't even need to check for contiguity, \emph{existence alone guarantees contiguity}, which in turn guarantees equality at homology level.

\begin{proposition}
\label{prop:all_refinements_contiguous}
Let $\mathcal{U} = \{U_i\}_{i \in I}$ and $\mathcal{V} = \{V_j\}_{j \in J}$ be two covers of a space $X$.
Any two refinement maps $f,g:\mathcal{V}\to \mathcal U$ are contiguous.
\end{proposition}

\begin{proof}
    Let $V_j\in \mathcal V$, since $f$ and $g$ are refinement maps there exist $U_{f(j)}\in \mathcal{U}$ and $U_{g(j)}\in \mathcal{U}$ such that $V_j \subseteq U_{f(j)}$ and $V_j \subseteq U_{g(j)}$ and therefore $V_j\subseteq U_{f(j)} \cap U_{g(j)}$. Therefore $f,g$ are contiguous.
\end{proof}

Consequently any diagram of persistent covers consisting exclusively of refinement maps commutes up to contiguity.

\subsection{Approximation up to Contiguity}
We have shown that any diagram of persistent covers consisting exclusively of refinement maps commutes up to contiguity.
This in turn allows us to relax the strict commutativity of interleavings and approximations to commutativity up to contiguity, which the cover level reduces to a simple existence check.
We now redefine the notion of approximation to accommodate this flexibility.

\begin{definition}[$c$-Approximation up to Contiguity]\label{def:c-approximation-contiguity}Let $\mathcal{U}, \mathcal{V}: (\mathbb{R}_+, \leq) \to \mathbf{Cov}(X)$ be two persistent covers.
We say that $\mathcal{U}$ and $\mathcal{V}$ are \textbf{$c$-approximated up to contiguity} if there exist families of refinement maps:
$$\phi_r: \mathcal{U}(r) \to \mathcal{V}(cr) \quad \text{and} \quad \psi_r: \mathcal{V}(r) \to \mathcal{U}(cr)$$

\end{definition}

This definition plays along well with the Proposition \ref{prop:all_refinements_contiguous}.
Given two persistent covers $\mathcal{U}$ and $\mathcal{V}$ it is enough to show the \textit{existence} of the families of refinements as described in Definition \ref{def:c-approximation-contiguity} in order to have the contiguity conditions automatically satisfied.


\begin{theorem}[Propagation of Approximations]
\label{thm:main_propagation}
Let $\mathcal{U}$ and $\mathcal{V}$ be persistent covers that are c-approximated up to contiguity.
Then the following persistence modules are all c-approximated:
\begin{align}
H_*\Nerve\;\mathcal{U} &\overset{c}{\longleftrightarrow} H_*\Nerve\;\mathcal{V} \label{eq:nrv_nrv}\tag{1}\\
H_*\CoNerve\;\mathcal{U} &\overset{c}{\longleftrightarrow} H_*\CoNerve\;\mathcal{V} \label{eq:conrv_conrv}\tag{2}\\
H_*\Nerve\;\mathcal{U} &\overset{c}{\longleftrightarrow} H_*\CoNerve\;\mathcal{V} \label{eq:mixed}\tag{3}
\end{align}
\end{theorem}

\begin{proof}
Since $\mathcal{U}$ and $\mathcal V$ are persistent covers then for $r\leq s$ there exist refinement maps $\mathcal U(r)\to \mathcal U(s)$ and $\mathcal V(r)\to \mathcal V(s)$.
By assumption $\mathcal U$ and $\mathcal V$ are approximated up to contiguity so there exists two families of refinement maps:
$$\phi_r: \mathcal{U}(r) \to \mathcal{V}(cr) \quad \text{and} \quad \psi_r: \mathcal{V}(r) \to \mathcal{U}(cr)\;.$$
By Proposition \ref{prop:all_refinements_contiguous} the following diagrams commute up to contiguity:
    \[
    \begin{tikzcd}
        \mathcal{U}(s) \arrow[r] \arrow[d, "\phi_s"'] & \mathcal{U}(t) \arrow[d, "\phi_t"] \\
        \mathcal V(cs) \arrow[r] & \mathcal V(ct)
    \end{tikzcd}
    \qquad \qquad
    \begin{tikzcd}
        \mathcal V(s) \arrow[r] \arrow[d, "\psi_s"'] & \mathcal V(t) \arrow[d, "\psi_t"] \\
        \mathcal{U}(cs) \arrow[r] & \mathcal{U}(ct)
    \end{tikzcd}
    \]
    
    \[
    \begin{tikzcd}[column sep=tiny]
        \mathcal{U}(t) \arrow[rr] \arrow[dr, "\phi_t"'] & & \mathcal{U}(c^2t) \\
        & \mathcal V(ct) \arrow[ur, "\psi_{ct}"'] &
    \end{tikzcd}
    \qquad \qquad
    \begin{tikzcd}[column sep=tiny]
        \mathcal V(t) \arrow[rr] \arrow[dr, "\psi_t"'] & & \mathcal V(c^2t) \\
        & \mathcal{U}(ct) \arrow[ur, "\phi_{ct}"'] &
    \end{tikzcd}
    \]
\textbf{1. Nerve and Co-Nerve (\ref{eq:nrv_nrv} and \ref{eq:conrv_conrv}):}
Let $F$ denote either the $H_*\Nerve(-)$ or $H_*\CoNerve(-)$ functor.
We construct the algebraic approximation by applying $F$ to the families of refinement maps $\phi$ and $\psi$.

The previous diagrams for $\phi$ and $\psi$ commute up to contiguity.
Lemma \ref{lem:homology_invariance_contiguity} establishes that contiguous refinement maps induce identical linear maps on homology.
Therefore, applying $F$ converts the contiguity relations directly into strict equalities:
$$
     F(\phi_s \circ \mathcal{U}(r \le s))= F(\mathcal{V}(cr \le cs) \circ \phi_r)
$$
and likewise:
\begin{align*}
F(\psi_{cr} \circ \phi_r)&=F(\mathcal{U}(r\leq c^2r))\;,\\
 F(\phi_{cr} \circ \psi_r )&=F( \mathcal{V}(r \leq c^2r))\;.
\end{align*}
Thus, the persistence modules are $c$-approximated.

\textbf{2. Mixed Case (\ref{eq:mixed}):}
We rely on the Functorial Dowker Theorem \cite{virk2019rips}, which provides a natural isomorphism of persistence modules $\theta: H_*\Nerve(\mathcal{V}) \xrightarrow{\cong} H_*\CoNerve(\mathcal{V})$.
From step 1, we established the approximation $H_*\Nerve(\mathcal{U}) \xleftrightarrow{c}H_*\Nerve(\mathcal{V})$.

Since approximation is transitive with respect to natural isomorphism (Lemma \ref{lem:iso-transport}), composing the approximation maps with the isomorphism $\theta$ yields the approximation $H_*\Nerve(\mathcal{U}) \xleftrightarrow{c} H_*\CoNerve(\mathcal{V})$.
\end{proof}

This theorem establishes that approximations between covers automatically induce approximations between all associated persistence modules, regardless of whether we use nerve or co-nerve constructions.

\begin{remark}
Our setting combines two existing relaxations of the standard interleaving.
\citeauthor{blumberg2022universalityhomotopyinterleavingdistance}~\cite{blumberg2022universalityhomotopyinterleavingdistance} introduced \emph{homotopy interleavings}, in which the coherence diagrams
commute up to homotopy and descend to strict algebraic interleavings under $H_*$.
\citeauthor{desilva2018theoryinterleavingscategoriesflow}~\cite{desilva2018theoryinterleavingscategoriesflow} generalized strict interleavings to arbitrary categories equipped with a \emph{flow}.
Our framework sits at the intersection: $\Cat{Cov}(X)$ carries a refinement flow, and Definition~\ref{def:c-approximation-contiguity} is a relaxed interleaving in this flow, with contiguity playing the role of homotopy.

A more abstract perspective would equip $\Cat{Cov}(X)$ with a 2-categorical structure in which contiguity is a 2-isomorphism and $\Nerve$, $\CoNerve$ are 2-functors sending these 2-cells to chain homotopies.
We have chosen to work directly with refinement maps and pairwise contiguity because the arguments at this level are sufficient for our approach while being intuitive and accessible.
\end{remark}

\subsection{Application: \v Cech and Vietoris-Rips Filtrations}

We now demonstrate how the classical approximation between Čech and Vietoris-Rips filtrations emerges naturally from our framework.

\subsubsection{Ball Covers and Čech Filtration}

\begin{definition}[Ball Cover]
For a metric space $(X,d)$ and $r>0$, the \textbf{ball cover} is:
$$\mathcal{B}(r)=\{B(x,r)\mid x\in X\}$$
where $B(x,r) = \{y \in X : d(x,y) \leq r\}$.
\end{definition}

Let $\mathcal{B}(s)=\{B(x_i,s)\}_{i\in I}$ and $\mathcal{B}(t)=\{B(x_i,t)\}_{i\in I}$ be ball covers of $X$ given two values $s,t\geq0$.
If $s\leq t$ we have that $B(x_i,s)\subseteq B(x_i,t)$ for all $i\in I$.
Let $f_{st}:I \to I$ be the identity map $f_{st}(i)=i$. 
Then $f_{st}$ is a refinement map and $\mathcal{B}(s)$ is a refinement of $\mathcal{B}(t)$ whenever $s\leq t$.

\begin{proposition}
The ball cover defines a persistent cover $\mathcal{B}: \poset \to \Cat{Cov}(X)$, with refinement maps induced by the inclusions $B(x_i,r) \subseteq B(x_i,s)$ for $r \leq s$.
\end{proposition}
\begin{proof}
    This is immediate from the inclusions $B(x,r) \subseteq B(x,s)$.
\end{proof}

The \textbf{Čech filtration} is precisely the functor $\Nerve \;\mathcal{B}$.
The nested structure of the \v Cech filtration is a direct consequence of the underlying cover.
The inclusion  $\Nerve \; \mathcal{B}(s) \subseteq \Nerve \;\mathcal{B}(t)$ for $s\leq t$ comes from the injectivity of the refinement map $\mathcal{B}(s) \to \mathcal{B}(t)$ by Lemma \ref{lem:refinement-inclusion}.
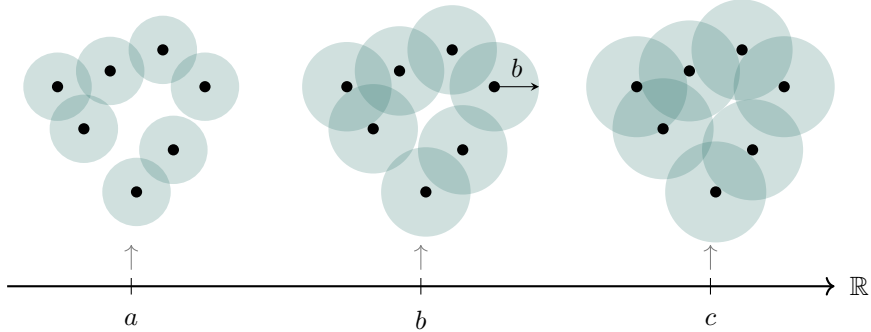
\begin{figure}[htb]
    \centering
    \newsavebox{\BallCoverOne}
\savebox{\BallCoverOne}{
\begin{tikzpicture}[
scale=0.7,
  ball_cover/.style={fill=color_r20g101b93, opacity=0.2},
  node/.style={circle, fill=black, inner sep=1.50pt}]
  \definecolor{color_r20g101b93}{rgb}{0.078,0.396,0.365}
  \begin{scope}
    \fill[ball_cover]
      (1.000,3.000) circle (0.650);
  \end{scope}
  \begin{scope}
    \fill[ball_cover]
      (2.000,3.300) circle (0.650);
  \end{scope}
  \begin{scope}
    \fill[ball_cover]
      (1.500,2.200) circle (0.650);
  \end{scope}
  \begin{scope}
    \fill[ball_cover]
      (3.000,3.700) circle (0.650);
  \end{scope}
  \begin{scope}
    \fill[ball_cover]
      (3.800,3.000) circle (0.650);
  \end{scope}
  \begin{scope}
    \fill[ball_cover]
      (3.200,1.800) circle (0.650);
  \end{scope}
  \begin{scope}
    \fill[ball_cover]
      (2.500,1.000) circle (0.650);
  \end{scope}
  \node[node] at (1.000,3.000) {};
  \node[node] at (2.000,3.300) {};
  \node[node] at (1.500,2.200) {};
  \node[node] at (3.000,3.700) {};
  \node[node] at (3.800,3.000) {};
  \node[node] at (3.200,1.800) {};
  \node[node] at (2.500,1.000) {};
\end{tikzpicture}

}
\newsavebox{\BallCoverTwo}
\savebox{\BallCoverTwo}{
\begin{tikzpicture}[
scale=0.7,
  ball_cover/.style={fill=color_r20g101b93, opacity=0.2},
  node/.style={circle, fill=black, inner sep=1.50pt}]
  \definecolor{color_r20g101b93}{rgb}{0.078,0.396,0.365}
  
  
  \begin{scope}
    \fill[ball_cover] (1.000,3.000) circle (0.850);
  \end{scope}
  \begin{scope}
    \fill[ball_cover] (2.000,3.300) circle (0.850);
  \end{scope}
  \begin{scope}
    \fill[ball_cover] (1.500,2.200) circle (0.850);
  \end{scope}
  \begin{scope}
    \fill[ball_cover] (3.000,3.700) circle (0.850);
  \end{scope}
  
  \begin{scope}
    \fill[ball_cover] (3.800,3.000) circle (0.850);
    \draw[->, >=stealth, black] (3.800,3.000) -- ++(0:0.850) node[midway, above] {$b$};
  \end{scope}
  
  \begin{scope}
    \fill[ball_cover] (3.200,1.800) circle (0.850);
  \end{scope}
  \begin{scope}
    \fill[ball_cover] (2.500,1.000) circle (0.850);
  \end{scope}
  
  \node[node] at (1.000,3.000) {};
  \node[node] at (2.000,3.300) {};
  \node[node] at (1.500,2.200) {};
  \node[node] at (3.000,3.700) {};
  \node[node] at (3.800,3.000) {};
  \node[node] at (3.200,1.800) {};
  \node[node] at (2.500,1.000) {};
\end{tikzpicture}

}
\newsavebox{\BallCoverThree}
\savebox{\BallCoverThree}{
\begin{tikzpicture}[
scale=0.7,
  ball_cover/.style={fill=color_r20g101b93, opacity=0.2},
  node/.style={circle, fill=black, inner sep=1.50pt}]
  \definecolor{color_r20g101b93}{rgb}{0.078,0.396,0.365}
  \begin{scope}
    \fill[ball_cover]
      (1.000,3.000) circle (0.960);
  \end{scope}
  \begin{scope}
    \fill[ball_cover]
      (2.000,3.300) circle (0.960);
  \end{scope}
  \begin{scope}
    \fill[ball_cover]
      (1.500,2.200) circle (0.960);
  \end{scope}
  \begin{scope}
    \fill[ball_cover]
      (3.000,3.700) circle (0.960);
  \end{scope}
  \begin{scope}
    \fill[ball_cover]
      (3.800,3.000) circle (0.960);
  \end{scope}
  \begin{scope}
    \fill[ball_cover]
      (3.200,1.800) circle (0.960);
  \end{scope}
  \begin{scope}
    \fill[ball_cover]
      (2.500,1.000) circle (0.960);
  \end{scope}
  \node[node] at (1.000,3.000) {};
  \node[node] at (2.000,3.300) {};
  \node[node] at (1.500,2.200) {};
  \node[node] at (3.000,3.700) {};
  \node[node] at (3.800,3.000) {};
  \node[node] at (3.200,1.800) {};
  \node[node] at (2.500,1.000) {};
\end{tikzpicture}

}
  
\begin{tikzpicture}[scale=1.1]
  \draw[thick, ->] (0,0) -- (10,0);
  \node at (10.3,0) {$\mathbb{R}$};
  
  \foreach \x/\label in {1.5/a, 5/b, 8.5/c} {
    \draw (\x,-0.1) -- (\x,0.1);
    \node at (\x,-0.4) {$\label$};
  }
  
  \foreach \x in {1.5, 5, 8.5} {
    \draw[->, gray] (\x,0.2) -- (\x,0.5);
  }

  \node at (1.5, 2){\usebox{\BallCoverOne}};

  \node at (5, 2) {\usebox{\BallCoverTwo}};
 
  \node at (8.5, 2) {\usebox{\BallCoverThree}};

\end{tikzpicture}
    \caption{The ball cover}
    \label{fig:ball_cover}
\end{figure}

Furthermore the Nerve and the Co-Nerve of the ball cover are the same simplicial complex.
This self-duality is a consequence of two properties: the index set coincides with the space ($I = X$), and the distance is symmetric.
\begin{proposition}\label{prop:ball_self_dual}
$$\CoNerve(\mathcal{B}(r)) = \Nerve(\mathcal{B}(r)) \quad \forall r \geq 0$$
\end{proposition}

\begin{proof}
Since the index set of $\mathcal{B}(r)$ is $X$ itself, both
complexes have vertex set $X$.  For a finite subset
$\sigma \subseteq X$:
\begin{align*}
    \sigma \in \CoNerve(\mathcal{B}(r))
    &\iff \exists\, y \in X \text{ such that }
          \sigma \subseteq B(y,r) \\
    &\iff \exists\, y \in X \text{ such that }
          \forall x \in \sigma,\; d(x,y) \le r \\
    &\iff \exists\, y \in X \text{ such that }
          \forall x \in \sigma,\; y \in B(x,r)
          \quad (\text{symmetry of } d) \\
    &\iff \textstyle\bigcap_{x \in \sigma} B(x,r) \neq \emptyset \\
    &\iff \sigma \in \Nerve(\mathcal{B}(r)). \qedhere
\end{align*}
\end{proof}

\subsubsection{Maximal-Clique Covers and Vietoris-Rips Filtration}

\begin{definition}
\label{def:maximal_clique_cover}
Let $(X,d)$ be a metric space and $r>0$. A subset $M \subseteq X$ is an \textbf{$r$-clique} if $\diam(M)\le r$. An $r$-clique is \textbf{maximal} if it is not properly contained in any other $r$-clique. The \textbf{maximal-clique cover} is the collection
\[
\mathcal{M}(r) = \{\,M : M \text{ is a maximal } r\text{-clique}\,\}.
\]
\end{definition}

This is the definition of \emph{maximal linked sets} in \cite[Sec.~8.1]{jardine1971mathematical}, but phrased in terms of diameters and covers as above.

\begin{lemma}
\label{lem:maximal_clique_refinement}
If $s \leq t$, then $\mathcal{M}(s)$ refines $\mathcal{M}(t)$.
\end{lemma}

\begin{proof}
Let $M \in \mathcal{M}(s)$. By definition, $\diam(M) \leq s \leq t$, so $M$ is a valid $t$-clique.

We consider two cases regarding the maximality of $M$ at scale $t$:
\begin{itemize}
    \item If $M$ is already maximal at scale $t$, then $M \in \mathcal{M}(t)$, and it is trivially contained in itself.
    \item If $M$ is not maximal at scale $t$, it must be properly contained in some larger maximal $t$-clique $N \in \mathcal{M}(t)$.
\end{itemize}
In either case, there exists a set $N \in \mathcal{M}(t)$ such that $M \subseteq N$. We define a refinement map $f: \mathcal{M}(s) \to \mathcal{M}(t)$ by selecting one such $N$ for each $M$.
\end{proof}
While Lemma \ref{lem:maximal_clique_refinement} guarantees the existence of a refinement map, the map is not uniquely determined because a maximal $s$-clique may be contained in the intersection of multiple distinct maximal $t$-cliques.
This non-uniqueness prevents $\mathcal{M}$ from being a functor from the poset category $\poset \to \Cat{Cov}(X)$, as independent choices of refinement maps satisfy composition only up to contiguity.
We can introduce a total order on $X$ that allows us to make a canonical choice without any issues knowing that any other choice results in a contiguous refinement.

\begin{figure}[htb]
    \centering
    \newsavebox{\MaxCliqueOne}
\savebox{\MaxCliqueOne}{
\begin{tikzpicture}[
  clique_cover/.style={fill=color_r20g101b93, opacity=0.2},
  edge41/.style={color_r164g179b178, opacity=0.8, thick},
  node/.style={circle, fill=black, inner sep=1.50pt}]
  \definecolor{color_r164g179b178}{rgb}{0.643,0.702,0.698}
  \definecolor{color_r20g101b93}{rgb}{0.078,0.396,0.365}
  \fill[clique_cover] (2.850,1.400) circle (0.582);
  \fill[clique_cover] (1.500,2.833) circle (0.734);
  \fill[clique_cover] (2.500,3.500) circle (0.589);
  \fill[clique_cover] (3.400,3.350) circle (0.582);
  \draw[edge41] (1.000,3.000) -- (2.000,3.300);
  \draw[edge41] (1.000,3.000) -- (1.500,2.200);
  \draw[edge41] (2.000,3.300) -- (1.500,2.200);
  \draw[edge41] (2.000,3.300) -- (3.000,3.700);
  \draw[edge41] (3.000,3.700) -- (3.800,3.000);
  \draw[edge41] (3.200,1.800) -- (2.500,1.000);
  \draw[<->, >=stealth, black] (3.200,1.800) -- (2.500,1.000) node[midway,above left] {$a$};
  \node[node] at (1.000,3.000) {};
  \node[node] at (2.000,3.300) {};
  \node[node] at (1.500,2.200) {};
  \node[node] at (3.000,3.700) {};
  \node[node] at (3.800,3.000) {};
  \node[node] at (3.200,1.800) {};
  \node[node] at (2.500,1.000) {};
\end{tikzpicture}
}
\newsavebox{\MaxCliqueTwo}
\savebox{\MaxCliqueTwo}{

\begin{tikzpicture}[
  clique_cover/.style={fill=color_r20g101b93, opacity=0.2},
  edge42/.style={color_r164g179b178, opacity=0.8, thick},
  node/.style={circle, fill=black, inner sep=1.50pt}]
  \definecolor{color_r164g179b178}{rgb}{0.643,0.702,0.698}
  \definecolor{color_r20g101b93}{rgb}{0.078,0.396,0.365}
  \fill[clique_cover] (2.000,1.600) circle (0.831);
  \fill[clique_cover] (2.850,1.400) circle (0.582);
  \fill[clique_cover] (1.500,2.833) circle (0.734);
  \fill[clique_cover] (2.500,3.500) circle (0.589);
  \fill[clique_cover] (3.400,3.350) circle (0.582);
  \fill[clique_cover] (3.500,2.400) circle (0.721);
  \draw[edge42] (1.000,3.000) -- (2.000,3.300);
  \draw[edge42] (1.000,3.000) -- (1.500,2.200);
  \draw[edge42] (2.000,3.300) -- (1.500,2.200);
  \draw[edge42] (2.000,3.300) -- (3.000,3.700);
  \draw[edge42] (1.500,2.200) -- (2.500,1.000);
  \draw[edge42] (3.000,3.700) -- (3.800,3.000);
  \draw[edge42] (3.800,3.000) -- (3.200,1.800);
  \draw[edge42] (3.200,1.800) -- (2.500,1.000);
  \node[node] at (1.000,3.000) {};
  \node[node] at (2.000,3.300) {};
  \node[node] at (1.500,2.200) {};
  \node[node] at (3.000,3.700) {};
  \node[node] at (3.800,3.000) {};
  \node[node] at (3.200,1.800) {};
  \node[node] at (2.500,1.000) {};
\end{tikzpicture}
}
\newsavebox{\MaxCliqueThree}
\savebox{\MaxCliqueThree}{
\begin{tikzpicture}[
  clique_cover/.style={fill=color_r20g101b93, opacity=0.2},
  edge43/.style={color_r164g179b178, opacity=0.8, thick},
  node/.style={circle, fill=black, inner sep=1.50pt}]
  \definecolor{color_r164g179b178}{rgb}{0.643,0.702,0.698}
  \definecolor{color_r20g101b93}{rgb}{0.078,0.396,0.365}
  \fill[clique_cover] (2.400,1.667) circle (1.096);
  \fill[clique_cover] (1.500,2.833) circle (0.734);
  \fill[clique_cover] (2.933,3.333) circle (0.984);
  \fill[clique_cover] (3.333,2.833) circle (1.092);
  \draw[edge43] (1.000,3.000) -- (2.000,3.300);
  \draw[edge43] (1.000,3.000) -- (1.500,2.200);
  \draw[edge43] (2.000,3.300) -- (1.500,2.200);
  \draw[edge43] (2.000,3.300) -- (3.000,3.700);
  \draw[edge43] (2.000,3.300) -- (3.800,3.000);
  \draw[edge43] (1.500,2.200) -- (3.200,1.800);
  \draw[edge43] (1.500,2.200) -- (2.500,1.000);
  \draw[edge43] (3.000,3.700) -- (3.800,3.000);
  \draw[edge43] (3.000,3.700) -- (3.200,1.800);
  \draw[edge43] (3.800,3.000) -- (3.200,1.800);
  \draw[edge43] (3.200,1.800) -- (2.500,1.000);
  \node[node] at (1.000,3.000) {};
  \node[node] at (2.000,3.300) {};
  \node[node] at (1.500,2.200) {};
  \node[node] at (3.000,3.700) {};
  \node[node] at (3.800,3.000) {};
  \node[node] at (3.200,1.800) {};
  \node[node] at (2.500,1.000) {};
\end{tikzpicture}

}
\begin{tikzpicture}[scale=1.1]
  \draw[thick, ->] (0,0) -- (10,0);
  \node at (10.3,0) {$\mathbb{R}$};
  
  \foreach \x/\label in {1.5/a, 5/b, 8.5/c} {
    \draw (\x,-0.1) -- (\x,0.1);
    \node at (\x,-0.4) {$\label$};
  }
  
  \foreach \x in {1.5, 5, 8.5} {
    \draw[->, gray] (\x,0.2) -- (\x,0.5);
  }

  \node at (1.5, 2){\usebox{\MaxCliqueOne}};

  \node at (5, 2) {\usebox{\MaxCliqueTwo}};
 
  \node at (8.5, 2) {\usebox{\MaxCliqueThree}};

\end{tikzpicture}
    \caption{The maximal clique cover. To help visualize the clique structure we added grey lines matching points $x,y$ whenever $d(x,y)\leq r$}
    \label{fig:max_clique_cover}
\end{figure}
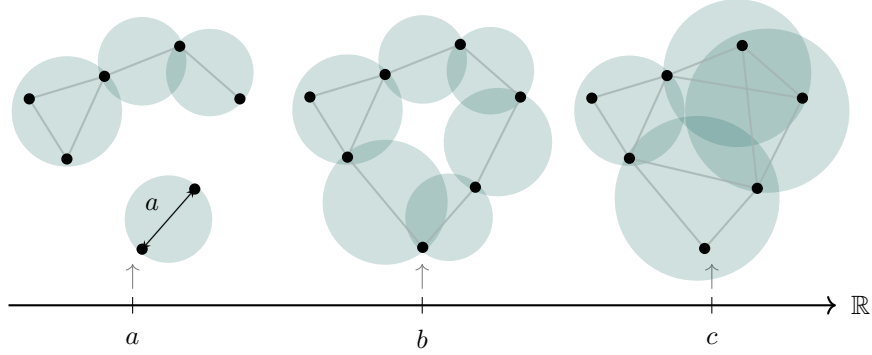

\begin{proposition}
Let $(X,d)$ be a metric space equipped with a total order.
Then $\mathcal{M}$ defines a covariant functor from the poset category $\poset\to\Cat{Cov}(X)$.
\end{proposition}

\begin{proof}
The total order on $X$ induces a lexicographic order on finite subsets of $X$. For $s \leq t$ and $M \in \mathcal{M}(s)$, define $f_{s,t}(M)$ to be the lexicographically minimal element of 
$$f_{s,t}(M)=\min\{N \in \mathcal{M}(t) : M \subseteq N\}$$

This set is non-empty by the previous lemma, and the lexicographic order ensures a unique minimal element exists.

\textbf{Identity:} For $M \in \mathcal{M}(r)$, we have $f_{r,r}(M) = M$ since $M$ is the unique maximal $r$-clique containing itself.
\textbf{Composition:} For $r \leq s \leq t$ and $M \in \mathcal{M}(r)$, let $N = f_{r,s}(M)$ and $P = f_{r,t}(M)$. We must show $f_{s,t}(N) = P$.

Since $M \subseteq N \subseteq P$ (as $N$ is a maximal $s$-clique containing $M$ and $P$ is a maximal $t$-clique containing $M$), we have that $P$ is a maximal $t$-clique containing $N$.

Suppose for contradiction that $f_{s,t}(N) = Q \neq P$ for some $Q \in \mathcal{M}(t)$ with $N \subseteq Q$. Since $Q$ is chosen as the lexicographically minimal maximal $t$-clique containing $N$, we have $Q <_{\text{lex}} P$.

But $M \subseteq N \subseteq Q$ implies $Q$ is a maximal $t$-clique containing $M$. Since $P = f_{r,t}(M)$ was chosen as the lexicographically minimal such clique, we must have $P \leq_{\text{lex}} Q$, contradiction.

Therefore $f_{s,t}(N) = P = f_{r,t}(M)$, establishing functoriality.
\end{proof}

\begin{definition}[Vietoris-Rips Complex ]
Let $(X, d)$ be a metric space and $r> 0$.
The \textbf{Vietoris-Rips} complex at scale $r$, denoted by $\Rips(r)$, has $X$ as its vertex set and a finite simplex $$\sigma \in \Rips(r) \iff\diam(\sigma) \leq r$$
\end{definition}

As pointed out by \citeauthor{virk2019rips} \cite{virk2019rips}, the Vietoris-Rips complex is the co-nerve of the maximal clique cover.
We expand this by showing that the \textbf{Vietoris-Rips filtration} is precisely the functor $\CoNerve \;\mathcal{M}$ .
Contrary to the case of the \v Cech filtration, for $s\leq t$ the refinement maps $\mathcal{M}(s)\to\mathcal{M}(t)$ are \textit{not} injective.
If one were instead to take the Nerve $\Nerve(\mathcal{M}(s)\to\mathcal{M}(t))$ it would not result in a nested sequence of simplicial complexes.
The nested structure of the Vietoris-Rips filtration comes from the the properties of the co-nerve functor which maps all refinements to simplicial inclusions (lemma \ref{lem:refinement-inclusion}).

\subsubsection{The Classical Approximation}
In this subsection we will show how the classic interleaving/approximation result between the persistent modules of the \v Cech and the Vietoris-Rips filtration appears as a consequence of the approximation result of their underlying covers.

\begin{proposition}[\cite{virk2019rips}]
\label{prop:clique_properties}
The maximal-clique cover $\mathcal{M}(r)$ is the unique cover satisfying:
$$\sigma \subseteq X \text{ is contained in some } M \in \mathcal{M}(r) \iff \diam(\sigma)\leq r$$
\end{proposition}

\begin{proposition}
\label{prop:ball_propeties}
    The Ball cover $\mathcal{B}(r)$ has the following properties:
\begin{enumerate}
    \item $\diam(\sigma)\leq r \implies \exists x_0\in X$ such that  $\sigma \subseteq B(x_0,r)$
    \item $\sigma \subseteq B(x_0,r)$ for some $x_0\in X \implies \diam(\sigma)\leq 2r$
\end{enumerate}
\end{proposition}

\begin{theorem}
\label{thm:ball-clique-approx}
There exist two families of refinement maps:
\begin{align*}
         f_r &: \mathcal{M}(r) \to\mathcal{B}(r) \tag{1}\\
        g_r &: \mathcal{B}(r) \to \mathcal{M}(2r) \tag{2}
\end{align*}that is,  $\mathcal{B}$ and $\mathcal{M}$ are $2$-approximated up to contiguity:
$$
\phi: \mathcal{B} \xleftrightarrow{2} \mathcal{M} : \psi
$$
\end{theorem}

\begin{proof}
The proof follows from the definitions and properties of the covers.
\begin{enumerate}
    \item  Let $\mathcal{M}(r)=\{M_i\}_{i\in I}$ for every $i\in I$ we have that $\diam(M_i)\leq r$ therefore there exists $x_j \in X$ such that $M_i \subseteq B(x_j,r)$. Let $f_r$ be the refinement that maps $f(i)=j$.
    \item Let $\mathcal{B}(r)=\{B(x_i,r)\}_{i\in I}$ then $\diam(B(x_i,r))\leq r$ which implies that there exists some $M_j \in \mathcal{M}(2r)$ such that $B(x_i,r) \subseteq M_j$. Let $g_r$ be the refinement that maps $g(i)=j$.
\end{enumerate}
\end{proof}

\begin{corollary}[\v Cech-Vietoris-Rips Approximation]
\label{cor:cech_rips}
There is a $2$-approximation between the persistence modules of the \v Cech and the Vietoris-Rips filtrations (among others):
$$H_*\Nerve\; \mathcal{B} \xleftrightarrow{2} H_*\CoNerve\;  \mathcal{M}$$
\end{corollary}

\begin{proof}
This follows immediately from Theorem \ref{thm:ball-clique-approx} and Theorem \ref{thm:main_propagation}.
\end{proof}

\section{Covers Quotiented by a Partition}
\label{sec:quotient}

Regardless of which simplicial complex is used, Nerve or Co-Nerve, the number of simplices grows exponentially with scale.
This exponential dependence is an artifact of unrestricted growth of the the underlying covers.
This still remains the dominant bottleneck in any state of the art computational setting.

Our aim is to tame the growth of the underlying covers in a controlled manner.
We do this by creating an opposing force.
Intuitively as the covers grow in size we want to quotient the space proportionately.
As we will see, what we obtain is exactly the original filtration quotient at each scale by a partition that gets coarser with the same scale.

\subsection{Clique Partition}

We will consider partitions that satisfy a single specific property: that the diameter of each block is bounded by the scale.
\begin{definition}[Partition]
A \textbf{partition} of a set $X$ is a family $P = \{C_i\}_{i \in I}$ of non-empty, pairwise disjoint subsets, called blocks, whose union is $X$.
\end{definition}

\begin{definition}[Category of Partitions]
Let $\Cat{Part}(X)$ be the category whose objects are partitions of $X$ and whose morphisms are refinement maps. A morphism $f: P \to Q$ (where $P = \{C_i\}_{i\in I}$ and $Q = \{D_j\}_{j\in J}$) is a function $f: I \to J$ such that $C_i \subseteq D_{f(i)}$ for all $i \in I$. We say $P$ refines $Q$.
\end{definition}

\begin{definition}[Persistent Clique Partition]
\label{def:persistent_clique_partition}
A \textbf{persistent clique partition} is a functor $P: \poset \to \Cat{Part}(X)$ such that for every $r > 0$:
\begin{enumerate}
    \item The partition $P(r)$, is a clique-partition, i.e. it consists of blocks with diameter at most $r$:
    $$ \forall C \in P(r), \quad \diam(C) \leq r $$
    \item If $s \leq t$, then $P(s)$ refines $P(t)$. That is, every block of $P(s)$ is contained in a block of $P(t)$.
\end{enumerate}
\end{definition}
This object has been called \textit{persistent set}~\cite{carlsson10acharacterization} but just like the case with covers we want to make sure to emphasize the extra conditions on the objects.
In the Section \ref{section:partitions} we will go in depth into these partitions and how to construct them but for now we will just assume that we have one persistent clique partition $P$.
\begin{figure}[htb]
    \centering
    \newsavebox{\PartitionOne}
\savebox{\PartitionOne}{
\begin{tikzpicture}[
  edge45/.style={color_r194g202b201, opacity=0.8, thick},
  partition/.style={black, opacity=0.8, thick, dashed, rounded corners=8pt},
  node/.style={circle, fill=black, inner sep=1.50pt}]
  \definecolor{color_r194g202b201}{rgb}{0.761,0.792,0.788}
  
  \draw[edge45] (1.000,3.000) -- (2.000,3.300);
  \draw[edge45] (1.000,3.000) -- (1.500,2.200);
  \draw[edge45] (2.000,3.300) -- (3.000,3.700);
  \draw[edge45] (3.000,3.700) -- (3.800,3.000);
  \draw[edge45] (3.200,1.800) -- (2.500,1.000);

  \draw[partition] (0.500, 3.100) -- (1.200, 3.500) -- (2.000, 2.200) -- (1.200, 1.800) -- cycle;
  \draw[partition] (1.600, 2.900) rectangle (2.400, 3.700);
  \draw[partition] (2.550, 3.600) -- (3.150, 4.200) -- (4.250, 3.100) -- (3.650, 2.500) -- cycle;
  \draw[partition] (3.050, 2.250) -- (3.650, 1.650) -- (2.650, 0.550) -- (2.050, 1.150) -- cycle;

  \node[node] at (1.000,3.000) {};
  \node[node] at (2.000,3.300) {};
  \node[node] at (1.500,2.200) {};
  \node[node] at (3.000,3.700) {};
  \node[node] at (3.800,3.000) {};
  \node[node] at (3.200,1.800) {};
  \node[node] at (2.500,1.000) {};
  
  \node[font=\small] at (0.4, 2.6) {$C_1$};
  \node[font=\small] at (2.0, 4.0) {$C_2$};
  \node[font=\small] at (3.9, 3.8) {$C_3$};
  \node[font=\small] at (3.3, 0.8) {$C_4$};
\end{tikzpicture}
}

\newsavebox{\PartitionTwo}
\savebox{\PartitionTwo}{
\begin{tikzpicture}[
  edge49/.style={color_r194g202b201, opacity=0.8, thick},
  partition/.style={black, opacity=0.8, thick, dashed, rounded corners=8pt},
  node/.style={circle, fill=black, inner sep=1.50pt}]
  \definecolor{color_r194g202b201}{rgb}{0.761,0.792,0.788}
  
  \draw[edge49] (1.000,3.000) -- (2.000,3.300);
  \draw[edge49] (1.000,3.000) -- (1.500,2.200);
  \draw[edge49] (2.000,3.300) -- (1.500,2.200);
  \draw[edge49] (2.000,3.300) -- (3.000,3.700);
  \draw[edge49] (2.000,3.300) -- (3.800,3.000);
  \draw[edge49] (1.500,2.200) -- (3.200,1.800);
  \draw[edge49] (1.500,2.200) -- (2.500,1.000);
  \draw[edge49] (3.000,3.700) -- (3.800,3.000);
  \draw[edge49] (3.000,3.700) -- (3.200,1.800);
  \draw[edge49] (3.800,3.000) -- (3.200,1.800);
  \draw[edge49] (3.200,1.800) -- (2.500,1.000);

  \draw[partition] (0.600, 3.100) -- (2.400, 3.600) -- (1.500, 1.700) -- cycle;
  \draw[partition] (2.550, 3.600) -- (3.150, 4.200) -- (4.250, 3.100) -- (3.650, 2.500) -- cycle;
  \draw[partition] (3.050, 2.250) -- (3.650, 1.650) -- (2.650, 0.550) -- (2.050, 1.150) -- cycle;

  \node[node] at (1.000,3.000) {};
  \node[node] at (2.000,3.300) {};
  \node[node] at (1.500,2.200) {};
  \node[node] at (3.000,3.700) {};
  \node[node] at (3.800,3.000) {};
  \node[node] at (3.200,1.800) {};
  \node[node] at (2.500,1.000) {};

  \node[font=\small] at (1.0, 3.6) {$C_1$};
  \node[font=\small] at (3.9, 3.8) {$C_2$};
  \node[font=\small] at (3.3, 0.8) {$C_3$};
\end{tikzpicture}
}

\newsavebox{\PartitionThree}
\savebox{\PartitionThree}{
\begin{tikzpicture}[
  edge51/.style={color_r194g202b201, opacity=0.8, thick},
  partition/.style={black, opacity=0.8, thick, dashed, rounded corners=8pt},
  node/.style={circle, fill=black, inner sep=1.50pt}]
  \definecolor{color_r194g202b201}{rgb}{0.761,0.792,0.788}
  
  \draw[edge51] (1.000,3.000) -- (2.000,3.300);
  \draw[edge51] (1.000,3.000) -- (1.500,2.200);
  \draw[edge51] (1.000,3.000) -- (3.000,3.700);
  \draw[edge51] (1.000,3.000) -- (3.200,1.800);
  \draw[edge51] (1.000,3.000) -- (2.500,1.000);
  \draw[edge51] (2.000,3.300) -- (1.500,2.200);
  \draw[edge51] (2.000,3.300) -- (3.000,3.700);
  \draw[edge51] (2.000,3.300) -- (3.800,3.000);
  \draw[edge51] (2.000,3.300) -- (3.200,1.800);
  \draw[edge51] (2.000,3.300) -- (2.500,1.000);
  \draw[edge51] (1.500,2.200) -- (3.000,3.700);
  \draw[edge51] (1.500,2.200) -- (3.800,3.000);
  \draw[edge51] (1.500,2.200) -- (3.200,1.800);
  \draw[edge51] (1.500,2.200) -- (2.500,1.000);
  \draw[edge51] (3.000,3.700) -- (3.800,3.000);
  \draw[edge51] (3.000,3.700) -- (3.200,1.800);
  \draw[edge51] (3.800,3.000) -- (3.200,1.800);
  \draw[edge51] (3.800,3.000) -- (2.500,1.000);
  \draw[edge51] (3.200,1.800) -- (2.500,1.000);

  \draw[partition] (0.600, 3.200) -- (2.100, 3.700) -- (3.600, 1.900) -- (2.500, 0.500) -- (1.100, 1.900) -- cycle;
  \draw[partition] (2.550, 3.600) -- (3.150, 4.200) -- (4.250, 3.100) -- (3.650, 2.500) -- cycle;

  \node[node] at (1.000,3.000) {};
  \node[node] at (2.000,3.300) {};
  \node[node] at (1.500,2.200) {};
  \node[node] at (3.000,3.700) {};
  \node[node] at (3.800,3.000) {};
  \node[node] at (3.200,1.800) {};
  \node[node] at (2.500,1.000) {};
  
  \node[font=\small] at (1.3, 1.1) {$C_1$};
  \node[font=\small] at (3.9, 3.8) {$C_2$};
\end{tikzpicture}
}

\begin{tikzpicture}[scale=1.1]
  \draw[thick, ->] (0,0) -- (10,0);
  \node at (10.3,0) {$\mathbb{R}$};
  
  \foreach \x/\label in {1.5/a, 5/b, 8.5/c} {
    \draw (\x,-0.1) -- (\x,0.1);
    \node at (\x,-0.4) {$\label$};
  }
  
  \foreach \x in {1.5, 5, 8.5} {
    \draw[->, gray] (\x,0.2) -- (\x,0.5);
  }

  \node at (1.5, 2) {\usebox{\PartitionOne}};
  \node at (5,   2) {\usebox{\PartitionTwo}};
  \node at (8.5, 2) {\usebox{\PartitionThree}};

\end{tikzpicture}
    \caption{A persistent clique-partition $P$. In this case it is a maximal clique partition (each partition is a maximal clique). We added all the edges whenever $d(x,y)\leq r$.  }
    \label{fig:clique_partition}
\end{figure}
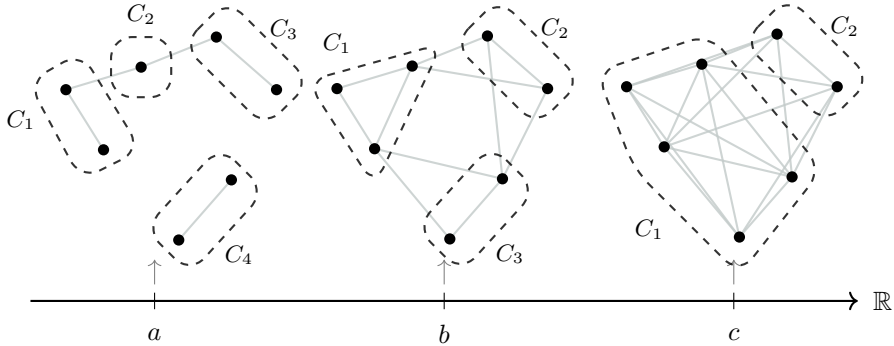

\subsection{Quotients of covers via pushforwards of relations}
The objective is to reduce the size of simplicial complexes in a controlled manner such that we guarantee an interleaving at the homology level.
We do this by considering the usual simplicial complexes quotient a partition.

Our first result shows that we can \textit{first} quotient the underlying cover and only after create a simplicial complex from it (via the nerve or co-nerve).
We show that this is equal to the simplicial complex of the original cover quotiented by the same partition.
This is fundamental for two reasons, first it allows us to focus only on the covers, second is computational since it tells us that we can compute the simplicial complex from a reduced space instead of first creating a full complex and then partitioning it.
Something that is very relevant when the computation of the simplicial complex itself is expensive.

\begin{definition}[Quotient Simplicial Complex]
Let $K$ be a simplicial complex on a vertex set $V$, and let $P$ be a partition of $V$ with projection map $q: V \to V/P$. The \textbf{quotient complex} $K/P$ is the simplicial complex on the vertex set $V/P$ defined as the image of the simplices of $K$ under $q$:
\[
K/P := \{ q(\sigma) \mid \sigma \in K \}.
\]
\end{definition}

\begin{proposition}
    Let $K$ be a simplicial complex on a vertex set $V$, and let $P$ be a partition of $V$ then the quotient simplicial complex $K/P$ a simplicial complex.
\end{proposition}
\begin{proof}
We must show that $K/P$ is closed under taking subsets.
Let $\tau \in K/P$ and let $\rho \subseteq \tau$ be a non-empty subset.
By definition, there exists a simplex $\sigma \in K$ such that $q(\sigma) = \tau$.
Consider the subset of $\sigma$ consisting of the preimages of $\rho$ restricted to $\sigma$:
\[
\sigma' = \sigma \cap q^{-1}(\rho) = \{ v \in \sigma \mid q(v) \in \rho \}.
\]
Since $K$ is a simplicial complex and $\sigma' \subseteq \sigma$, we have $\sigma' \in K$.
Because $q$ maps $\sigma$ surjectively onto $\tau$, it maps the restricted set $\sigma'$ surjectively onto $\rho$.
Thus, $q(\sigma') = \rho$, which implies $\rho \in K/P$.
\end{proof}

Recall the Dowker complexes from a relation $R \subseteq X \times Y$ be a relation:
\begin{align*}
C(R) &= \{\sigma \subseteq X : \exists\, y \in Y \text{ such that } (x, y) \in R \text{ for all } x \in \sigma\}, \\
C(R^{op}) &= \{\tau \subseteq Y : \exists\, x \in X \text{ such that } (y, x) \in R^{op} \text{ for all } y \in \tau\}.
\end{align*}

Let $P_X$ be a partition of $X$ with quotient map $q_X : X \to X/P_X$. We define the \textbf{pushforward relation} $R/P_X \subseteq (X/P_X) \times Y$ as:
\begin{definition}[Pushforward relation]
\[
R/P_X := (q_X \times \mathrm{id}_Y)(R) = \{([x], y) : \exists\, x' \in [x] \text{ such that } (x', y) \in R\}.
\]
\end{definition}

\begin{lemma}
\label{lem:quotient_dowker_relation}

Let $R \subseteq X \times Y$ be a binary relation and let $P_X$ be a partition of $X$. Then:
\begin{enumerate}
    \item $C(R)/P_X = C(R/P_X)$
    \item $C(R^{op}) \subseteq C((R/P_X)^{op})$
\end{enumerate}
\end{lemma}

\begin{proof}
\textbf{Part 1.} We prove both inclusions.

($\subseteq$) We observe that the projection preserves adjacency: if a set of vertices $\sigma$ share a common witness $y$, their equivalence classes share that same witness. Let $\tau \in C(R)/P_X$, so $\tau = q_X(\sigma)$ for some $\sigma \in C(R)$. There exists $y \in Y$ with $(x,y) \in R$ for all $x \in \sigma$. For each $[x] \in \tau$, pick $x' \in \sigma$ with $q_X(x') = [x]$. Then $(x', y) \in R$ and $x' \in [x]$, so $([x], y) \in R/P_X$. Therefore, $\tau \in C(R/P_X)$.

($\supseteq$) To lift a simplex from the quotient back to $X$, we select representatives from each class $[x] \in \tau$ that are guaranteed to connect to the witness $y$. Let $\tau \in C(R/P_X)$. There exists $y \in Y$ such that $([x], y) \in R/P_X$ for all $[x] \in \tau$. For each $[x] \in \tau$, choose $x' \in [x]$ with $(x', y) \in R$. Let $\sigma = \{x' \mid [x] \in \tau\}$. Then $\sigma \in C(R)$ and $q_X(\sigma) = \tau$, so $\tau \in C(R)/P_X$.

\textbf{Part 2.} This inclusion holds because quotienting the domain relaxes the witness condition for the codomain: elements of a simplex $\tau \subseteq Y$ no longer need to connect to the exact same vertex in $X$, but merely to the same equivalence class in $X/P_X$.
Let $\tau \in C(R^{op})$. By definition, $\exists x \in X$ such that $(y, x) \in R^{op}$ for all $y \in \tau$, which means $(x,y) \in R$. Let $[x]$ be the equivalence class of $x$ in $X/P_X$. Because $x \in [x]$ and $(x,y) \in R$, we have $([x], y) \in R/P_X$ by the definition of the pushforward. This implies $(y, [x]) \in (R/P_X)^{op}$ for all $y \in \tau$. Hence $\tau \in C((R/P_X)^{op})$, proving the inclusion.
\end{proof}


We now apply this to a cover $\mathcal{U} = \{U_i\}_{i \in I}$ of $X$ with inclusion relation $R_{\mathcal{U}} \subseteq X \times I$. We may partition either the space $X$ or the index set $I$.

\begin{definition}[Quotient Covers]
\label{def:quotient_covers}
Let $P_X$ and $P_I$ be partitions of $X$ and $I$ respectively.
\begin{enumerate}
    \item The \textbf{quotient cover on the space} is the cover of $X/P_X$ defined by the images under $q_X$:
    \[ \mathcal{U}/ P_X := \{ q_X(U_i) \mid i\in I \}. \quad \text{where} \quad q_X(U_i) = \{[x] : x \in U_i\}\]
    \item The \textbf{quotient cover on the index set} is the cover of $X$ defined by unions over classes in $I/P_I$:
    \[ \mathcal{U} / P_I := \left\{ \bigcup_{i'\in[i]} U_{i'} \;\middle|\; [i] \in I/P_I \right\}. \]
\end{enumerate}
\end{definition}

It is straightforward to verify that the inclusion relations of these covers are exactly the pushforwards of the original relation $R_{\mathcal{U}}$:
\[
R_{\mathcal{U}/P_X} = R_{\mathcal{U}} / P_X 
\quad \text{and} \quad
R_{\mathcal{U}/P_I} = R_{\mathcal{U}} / P_I.
\]
Combining this observation with Lemma \ref{lem:quotient_dowker_relation} yields the following corollary.

\begin{corollary}[Quotients of Nerves and CoNerves]
\label{cor:quotient_of_covers}
For a cover $\mathcal{U}$ and partitions $P_X, P_I$:
\begin{align*}
    \mathrm{CoNerve}(\mathcal{U})/P_X &= \mathrm{CoNerve}(\mathcal{U}/P_X), \\
      \mathrm{Nerve}(\mathcal{U})/P_I&=\mathrm{Nerve}(\mathcal{U}/P_I).
\end{align*}
\end{corollary}

\begin{proof}
Recall that $CoNrv(\mathcal{U}) = C(R_{\mathcal{U}})$ and $Nrv(\mathcal{U}) = C(R_{\mathcal{U}}^{op})$.

For the first equality, we apply Lemma \ref{lem:quotient_dowker_relation}(1) to the relation $R_{\mathcal{U}}$ and partition $P_X$.
Since $R_{\mathcal{U}/P_X} = R_{\mathcal{U}}/P_X$, the complex of the quotient cover is exactly the quotient of the original complex: $C(R_{\mathcal{U}})/P_X = C(R_{\mathcal{U}}/P_X)$.

For the second equality, we apply Lemma \ref{lem:quotient_dowker_relation}(1) to the transpose relation $R_{\mathcal{U}}^{op} \subseteq I \times X$ and the index partition $P_I$.
This yields $C(R_{\mathcal{U}}^{op})/P_I = C(R_{\mathcal{U}}^{op}/P_I)$.
Recognizing that the pushforward of the transpose exactly defines the transpose relation of the index-quotient cover ($R_{\mathcal{U}}^{op}/P_I = R_{\mathcal{U}/P_I}^{op}$), we conclude $Nrv(\mathcal{U})/P_I = Nrv(\mathcal{U}/P_I)$.
\end{proof}

The impact of this result is a computational one.
If we want to consider a quotient complex we dont need to compute the entire complex and then quotient it.
This result tells us that we can simply compute the simplicial complex on a quotient cover.

\subsection{Pullback Covers}
Ideally we would now like to construct approximations between some persistent cover $\mathcal U$ and its quotient $\mathcal U/P_X$.
We cannot relate these two covers directly since $\mathcal{U}$ is a cover of $X$ while $\mathcal{U}/P_X$ is a cover of $ X/P_X$.
Furthermore the domain of the quotient cover, $X/P_X$, changes with parameter $s$.
Since the underlying space is not fixed, $\mathcal{U}/P_X$ does not form a persistent cover in the sense that the structure maps are not refinement maps.
To do this formally we must relate covers on $X$ to covers on $X/P_X$ using the pullback.
The pullback resolves the issue by lifting the evolving quotient geometry back to the stable space $X$, restoring the necessary functorial properties.

\begin{definition}[Pullback Cover]
\label{def:pullback_cover}
Let $q: X \to Y$ be a surjective map and $\mathcal{V} = \{V_i\}_{i \in I}$ be a cover of $Y$. The \textbf{pullback cover} $q^*(\mathcal{V})$ is a cover of $X$ defined by:
$$q^*(\mathcal{V}) = \{q^{-1}(V_i)\}_{i \in I}$$
\end{definition}

In a persistent setting, the pullback respects the refinement structure.

\begin{proposition}[Functoriality of the Pullback]
\label{prop:pullback-persistence}
Let $\mathcal{U}: \poset \to \Cat{Cov}(X)$ be a persistent cover with refinement maps $f_{s,t}: \mathcal{U}(s) \to \mathcal{U}(t)$ for $s \leq t$.
Let $P: \poset \to \Cat{Part}(X)$ be a persistent partition and $q_s: X \to X/P(s)$ denote the quotient maps.
The pullback cover $\mathcal{U}_P$ at scale $s$ as:
$$ \mathcal{U}_P(s) := q_s^* (\mathcal{U}(s)/P(s)) = \{ q_s^{-1}(q_s(U)) \mid U \in \mathcal{U}(s) \} $$
The collection $\mathcal{U}_P$ forms a persistent cover $\mathcal{U}_P: \poset \to \Cat{Cov}(X)$.
\end{proposition}

\begin{proof}
Geometrically, the set $q_s^{-1}(q_s(U))$ is the \textbf{saturation} of $U$ with respect to the partition $P(s)$; it consists of the union of all partition blocks that intersect $U$.
We must construct a refinement map $F_{s,t}: \mathcal{U}_P(s) \to \mathcal{U}_P(t)$ for any $s \leq t$ such that $V \subseteq F_{s,t}(V)$ for all $V \in \mathcal{U}_P(s)$.

Let $V \in \mathcal{U}_P(s)$. By definition, $V = q_s^{-1}(q_s(U))$ for some $U \in \mathcal{U}(s)$.
Since $\mathcal{U}$ is persistent, there exists a set $U' \in \mathcal{U}(t)$ such that $U \subseteq U'$.
We define $V' = q_t^{-1}(q_t(U')) \in \mathcal{U}_P(t)$ as the target of $V$.

We now show that $V \subseteq V'$.
Let $x \in V$. By definition, $q_s(x) \in q_s(U)$, which implies there exists some $u \in U$ such that $q_s(x) = q_s(u)$ (i.e., $x$ and $u$ are in the same block of $P(s)$).
Since $P$ is a persistent partition, $P(s)$ refines $P(t)$. This implies that if $x$ and $u$ are in the same block at scale $s$, they are in the same block at scale $t$.
Therefore, $q_t(x) = q_t(u)$.
Since $u \in U$ and $U \subseteq U'$, we have $u \in U'$.
Thus, $q_t(x) = q_t(u) \in q_t(U')$.
This implies $x \in q_t^{-1}(q_t(U')) = V'$.

Since $x$ was arbitrary, $V \subseteq V'$. The map sending the saturation of $U$ to the saturation of $f_{s,t}(U)$ defines a valid persistent structure.
\end{proof}

An immediate but crucial property of the pullback is that it preserves the intersection patterns of the cover, thereby preserving the Nerve complex.

\begin{lemma}[\cite{dey2017topological}]
\label{lem:pullback_nerve}
Let $q: X \to Y$ be a surjective map and let $\mathcal{V} = \{V_j\}_{j \in J}$ be a cover of $Y$. Then the pullback cover $q^*\mathcal{V} = \{q^{-1}(V_j)\}_{j \in J}$ satisfies:
\[
\mathrm{Nerve}(q^*\mathcal{V}) = \mathrm{Nerve}(\mathcal{V}).
\]
\end{lemma}

\begin{proof}
A simplex $\{j_1, \ldots, j_k\}$ lies in $\mathrm{Nerve}(\mathcal{V})$ if and only if $\bigcap_{\ell=1}^k V_{j_\ell} \neq \emptyset$. Similarly, it lies in $\mathrm{Nerve}(q^*\mathcal{V})$ if and only if $\bigcap_{\ell=1}^k q^{-1}(V_{j_\ell}) \neq \emptyset$. We show these conditions are equivalent.

\medskip
$(\Rightarrow)$ Suppose $x \in \bigcap_{\ell=1}^k q^{-1}(V_{j_\ell})$. Then $q(x) \in V_{j_\ell}$ for all $\ell$, so $q(x) \in \bigcap_{\ell=1}^k V_{j_\ell}$.

\medskip
$(\Leftarrow)$ Suppose $y \in \bigcap_{\ell=1}^k V_{j_\ell}$. Since $q$ is surjective, there exists $x \in X$ with $q(x) = y$. Then $x \in q^{-1}(V_{j_\ell})$ for all $\ell$, so $x \in \bigcap_{\ell=1}^k q^{-1}(V_{j_\ell})$.
\end{proof}

The simplicial isomorphisms between both Nerves, being the identity, renders the isomorphism natural with respect to the persistence structure.

\begin{lemma}[Persistence Isomorphism]
    \label{lem:pullback_natural}
    Let $\mathcal{U} : (\mathbb{R}_+, \leq) \to \mathbf{Cov}(X)$ be a persistent cover, and let $P$ be a persistent partition of $X$.
    Let $\mathcal{U}_P$ be the persistent pullback cover defined by $\mathcal{U}_P(s) := q_s^*(\mathcal{U}(s)/P(s))$.
    There is a natural isomorphism of persistence modules:
    \[
        \theta_r: H_*\Nerve\;\mathcal{U}_P (r)\xrightarrow{\cong} H_*\Nerve \;\mathcal{U}/P(r)
    \]
\end{lemma}
\begin{proof}
    Write $I_s$ for the index set of $\mathcal{U}(s)$. By construction (Definitions \ref{def:quotient_covers} and \ref{def:pullback_cover}), both $(\mathcal{U}/P)(s)$ and $\mathcal{U}_P(s)$ are also indexed by $I_s$, with elements $q_s(U_i)$ and $q_s^{-1}(q_s(U_i))$ respectively for $i \in I_s$.
    Applying Lemma \ref{lem:pullback_nerve} to the surjection $q_s : X \to X/P(s)$ and the cover $(\mathcal{U}/P)(s)$ yields the simplicial equality
    \[
        \Nerve(\mathcal{U}_P(s)) = \Nerve((\mathcal{U}/P)(s))
    \]
    as simplicial complexes on the common vertex set $I_s$.
    We take $\theta_s$ to be the identity induced on homology.
    
    For $r \leq s$, the structural maps on both sides of the diagram are induced by refinement maps between covers indexed over $I_r$ and $I_s$.
    Since any two refinement maps between two given covers are contiguous, they induce the same map on homology).
    The naturality square thus commutes, defining an isomorphism of persistence modules.
\end{proof}

This proposition bridges the theoretical gap between the fixed metric space $X$ and the varying quotient spaces.
It allows us to perform geometric analysis and prove interleavings using the pullback cover $\mathcal{U}_P$ (which resides on the fixed space $X$), while guaranteeing that these results apply to the computationally efficient quotient complex.

Furthermore, by applying Dowker Duality at every scale, this isomorphism extends to the Co-Nerve:
\[
    H_*\mathrm{CoNrv}(\mathcal{U}/P) \cong H_*\mathrm{Nrv}(\mathcal{U}/P) \cong H_*\mathrm{Nrv}(\mathcal{U}_P) \cong H_*\mathrm{CoNrv}(\mathcal{U}_P)
\]

So yet again, if we want to compute quotient simplicial complexes it is enough to look and compare the underlying covers.

\subsection{Quotient Maximal Clique Cover}
We wish to compare the standard maximal-clique filtration $\mathcal{M}$ with the quotient filtration $\mathcal{M}/P$, for a given persistent clique partition $P$.
We consider the pullback cover:
$$ \mathcal{M}_P(r) = q_r^*(\mathcal{M}/P)(r) = \{ q_r^{-1}(q_r(M))\mid M \in \mathcal{M}(r)\}$$

\begin{theorem}[$3$-Multiplicative Approximation]
\label{thm:mult-approx}
There exists a $3$-approximation up to contiguity between the maximal-clique cover $\mathcal{M}$ and $\mathcal{M}_P$,
That is, there exist refinement maps:
\begin{align*}
         f_r &: \mathcal{M}(r) \to\mathcal{M}_P(r) \tag{1}\\
        g_r &: \mathcal{M}_P(r) \to \mathcal{M}(3r) \tag{2}
\end{align*}
\end{theorem}

\begin{proof}
$\mathcal{M}(r) \to \mathcal{M}_P(r)$:
Let $M \in \mathcal{M}(r)$. The corresponding element in the pullback cover is $U = q_r^{-1}(q_r(M))$.
Since $M \subseteq q_r^{-1}(q_r(M))$ trivially, the identity map on the index set is a refinement map.

$\mathcal{M}_P(r) \to \mathcal{M}(3r)$:
Let $U \in \mathcal{M}_P(r)$. By definition, $U = q_r^{-1}(q_r(M))$ for some maximal $r$-clique $M$.
We bound the diameter of $U$. Let $x, y \in U$.
By definition of the inverse image, $q_r(x) \in q_r(M)$ and $q_r(y) \in q_r(M)$.
This implies that $x$ belongs to some partition block $C_x \in P(r)$ such that $C_x \cap M \neq \emptyset$, and similarly for $y$ with some block $C_y$.

Let $x' \in C_x \cap M$ and $y' \in C_y \cap M$.
By the triangle inequality:
$$ d(x, y) \leq d(x, x') + d(x', y') + d(y', y) $$
Using the properties of the partition and the clique:
\begin{itemize}
    \item $d(x, x') \leq r$ because $x, x'$ are in the same block $C_x \in P(r)$ and $\diam(C_x) \leq r$.
    \item $d(y', y) \leq r$ because $y, y'$ are in the same block $C_y \in P(r)$ and $\diam(C_y) \leq r$.
    \item $d(x', y') \leq r$ because $x', y' \in M$ and $M$ is an $r$-clique.
\end{itemize}
Thus, $\diam(U) \leq 3r$.

By Proposition \ref{prop:clique_properties}, $U$ is contained in some maximal $3r$-clique.
\end{proof}

This result on covers lifts immediately to the homology of the associated complexes.

\begin{corollary}
Let $P$ be a persistent clique partition. There is a $3$-approximation between the persistence modules of the Vietoris-Rips filtration and the quotient Vietoris-Rips filtration:
$$ H_*\CoNerve\; \mathcal{M} \xleftrightarrow{3} H_*\CoNerve \;(\mathcal{M}/P) $$
\end{corollary}

\begin{proof}
By Theorem \ref{thm:mult-approx}, we have the interleaving of covers $\mathcal{M}\xleftrightarrow{3} \mathcal{M}_P$.
Applying the homology functor $H_*\CoNerve(-)$ yields:
$$ H_*\CoNerve\mathcal{M} \xleftrightarrow{3} H_*\CoNerve \mathcal{M}_P$$
We then take the natural isomorphisms from the Dowker Duality and Lemma \ref{lem:pullback_natural}:
\begin{align*}
    H_*\CoNerve\mathcal{M}_P &\cong H_*\Nerve\mathcal{M}_P & (\text{Dowker}) \\
    &\cong H_*\Nerve(\mathcal{M}/P) & (\text{Lemma } \ref{lem:pullback_natural}) \\
    &\cong H_*\CoNerve(\mathcal{M}/P) & (\text{Dowker})
\end{align*}
from which we conclude:
$$
 H_*\CoNerve \mathcal M\xleftrightarrow{3} H_*\CoNerve(\mathcal{M}/P)
$$
\end{proof}

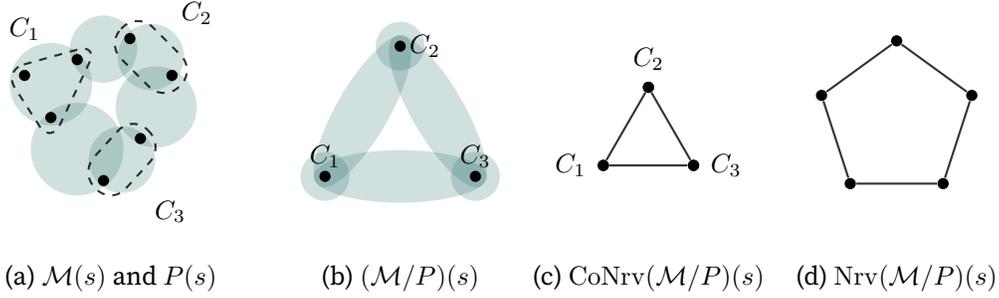
\begin{figure}[htb]
    \centering
    \newsavebox{\QuotientMaxOne}
\savebox{\QuotientMaxOne}{
\begin{tikzpicture}[
  clique_cover/.style={fill=color_r20g101b93, opacity=0.2},
  edge65/.style={color_r194g202b201, opacity=0.8, thick},
  edge66/.style={black, opacity=0.8, thick},
  partition/.style={black, opacity=0.8, thick, dashed, rounded corners=8pt},
  node/.style={circle, fill=black, inner sep=1.50pt},scale=0.7]
  
  \definecolor{color_r194g202b201}{rgb}{0.761,0.792,0.788}
  \definecolor{color_r20g101b93}{rgb}{0.078,0.396,0.365}
  
  \fill[clique_cover] (2.000,1.600) circle (0.881);
  \fill[clique_cover] (2.850,1.400) circle (0.632);
  \fill[clique_cover] (1.500,2.833) circle (0.784);
  \fill[clique_cover] (2.500,3.500) circle (0.639);
  \fill[clique_cover] (3.400,3.350) circle (0.632);
  \fill[clique_cover] (3.500,2.400) circle (0.771);

  \draw[partition] (0.600, 3.100) -- (2.400, 3.600) -- (1.500, 1.700) -- cycle;
  
  \draw[partition] (2.550, 3.600) -- (3.150, 4.200) -- (4.250, 3.100) -- (3.650, 2.500) -- cycle;
  
  \draw[partition] (3.050, 2.250) -- (3.650, 1.650) -- (2.650, 0.550) -- (2.050, 1.150) -- cycle;

  \node[node] at (1.000,3.000) {};
  \node[node] at (2.000,3.300) {};
  \node[node] at (1.500,2.200) {};
  \node[node] at (3.000,3.700) {};
  \node[node] at (3.800,3.000) {};
  \node[node] at (3.200,1.800) {};
  \node[node] at (2.500,1.000) {};
  
  \node[above] at (1.000, 3.500) {$C_1$};
  \node[above right] at (3.800, 3.800) {$C_2$};
  \node[below right] at (3.300, 0.800) {$C_3$};
\end{tikzpicture}
}

\newsavebox{\QuotientMaxTwo}
\savebox{\QuotientMaxTwo}{
\begin{tikzpicture}[
scale=0.5,
  ball_cover/.style={fill=color_r20g101b93, opacity=0.2},
  node/.style={circle, fill=black, inner sep=1.50pt}]
  \definecolor{color_r20g101b93}{rgb}{0.078,0.396,0.365}
  
  \coordinate (A) at (1.000, 1.000);
  \coordinate (B) at (5.000, 1.000);
  \coordinate (C) at (3.000, 4.464); 


  \begin{scope}
    \fill[ball_cover] 
      (3.000, 1.000) ellipse (2.3 and 0.7); 
  \end{scope}

  \begin{scope}
    \fill[ball_cover, rotate around={60:(2.000, 2.732)}] 
      (2.000, 2.732) ellipse (2.3 and 0.7);
  \end{scope}

  \begin{scope}
    \fill[ball_cover, rotate around={-60:(4.000, 2.732)}] 
      (4.000, 2.732) ellipse (2.3 and 0.7);
  \end{scope}

  \begin{scope}
    \fill[ball_cover] (A) circle (0.650);
  \end{scope}
  \begin{scope}
    \fill[ball_cover] (B) circle (0.650);
  \end{scope}
  \begin{scope}
    \fill[ball_cover] (C) circle (0.650);
  \end{scope}

  \node[node] at (A) {};
  \node[node] at (B) {};
  \node[node] at (C) {};

  \node[above] at (A) {$C_1$};
  \node[above] at (B) {$C_3$};
  \node[right] at (C) {$C_2$};
\end{tikzpicture}

}

\newsavebox{\QuotientMaxThree}
\savebox{\QuotientMaxThree}{
\begin{tikzpicture}[
 point/.style={circle, fill=black, inner sep=1.5pt},
 connection/.style={thick, opacity=0.8},
 scale=0.4
]
  \node[point] (A) at (0, 0) {}; 
  \node[point] (F) at (3, 0) {}; 
  \node[point] (H) at (1.5, 2.598) {}; 
   \foreach \label in {A, F, H} {
        \node[point] (\label) at (\label) {};
    }
    \node[above=3pt] at (H) {$C_2$};
    \node[left=3pt] at (A) {$C_1$};
    \node[right=3pt] at (F) {$C_3$};
    \draw[connection] (A) -- (F);
    \draw[connection] (F) -- (H);
    \draw[connection] (A) -- (H);
\end{tikzpicture}
}

\newsavebox{\QuotientMaxFour}
\savebox{\QuotientMaxFour}{
\begin{tikzpicture}[
 point/.style={circle, fill=black, inner sep=1.5pt},
 connection/.style={thick, opacity=0.8},
 scale=0.3 
]
  \node[point] (N1) at (90:3.5) {}; 
  \node[point] (N2) at (30:3.5) {};
  \node[point] (N3) at (330:3.5) {};
  \node[point] (N4) at (270:3.5) {};
  \node[point] (N5) at (210:3.5) {};
  \node[point] (N6) at (150:3.5) {};

  \fill[red!20] (N1.center) -- (N2.center) -- (N6.center) -- cycle;
  \fill[red!20] (N2.center) -- (N3.center) -- (N4.center) -- cycle;
  \fill[red!20] (N4.center) -- (N5.center) -- (N6.center) -- cycle;

  \draw[connection] (N1) -- (N2);
  \draw[connection] (N2) -- (N3);
  \draw[connection] (N3) -- (N4);
  \draw[connection] (N4) -- (N5);
  \draw[connection] (N5) -- (N6);
  \draw[connection] (N6) -- (N1);

  \draw[connection] (N2) -- (N4);
  \draw[connection] (N4) -- (N6);
  \draw[connection] (N6) -- (N2);
  
  \node[point] at (N1) {};
  \node[point] at (N2) {};
  \node[point] at (N3) {};
  \node[point] at (N4) {};
  \node[point] at (N5) {};
  \node[point] at (N6) {};

\end{tikzpicture}
}

\begin{tikzpicture}[scale=1.1]
  \node at (1.5, 0){\usebox{\QuotientMaxOne}};
  \node at (1.5,-2) {(a) $\mathcal M(s)$ and $P(s)$};
  \node at (5, 0) {\usebox{\QuotientMaxTwo}};
  \node at (5,-2) {(b) $(\mathcal{M}/P)(s)$};
  \node at (8, 0) {\usebox{\QuotientMaxThree}};
  \node at (8,-2) {(c) $\CoNerve(\mathcal M/P)(s)$};
  \node at (11, 0) {\usebox{\QuotientMaxFour}};
  \node at (11,-2) {(d) $\Nerve(\mathcal M/P)(s)$};
\end{tikzpicture}
    \caption{Example of an instance of the maximal clique cover $\mathcal M$ quotient some clique partition $P=\{C_1,C_2,C_3\}$ for a given scale $s>0$. \textbf{b)} shows the resulting cover while \textbf{c)} and \textbf{d)} show the simplicial complexes constructed by taking the co-nerve and nerve of the quotient cover.}
    \label{fig:quotient_max_clique}
\end{figure}

\subsection{Quotient Ball Cover}

Since the index set is the space itself, any partition $P_X$ of the space simultaneously partitions the index set.
This yields two quotient covers: $\mathcal{B}/P_X$, a cover of $X/P$ indexed by points, and $\mathcal{B}/P_I$, a cover of $X$ indexed by clusters.  The symmetry of~$d$ ensures that the inclusion relations of these two covers are
transposes of each other:
\[
  R_{\mathcal{B}/P_I} = R_{\mathcal{B}/P_X}^{op}.
\]
To see this, note that for any $x \in X$ and $[y] \in X/P_X$:
\begin{align*}
  x \in \textstyle\bigcup_{z \in [y]} B(z,r)
  &\iff \exists\, z \in [y] \text{ with } d(x,z) \leq r \\
  &\iff \exists\, z \in [y] \text{ with } d(z,x) \leq r \\
  &\iff [y] \cap B(x,r) \neq \emptyset,
\end{align*}
Because these relations are exact transposes, applying the Nerve and Co-Nerve functors yields a duality between the covers.

\begin{proposition}
Given a ball cover $\mathcal B=\{B(x,r)\}_{x \in X}$ at scale $r>0$ and a partition $P_X$ of $X$ we have:
\begin{align*}
    \CoNerve(\mathcal{B}/P_X) &= \Nerve(\mathcal{B}/P_I)\\
    \Nerve(\mathcal{B}/P_X) &= \CoNerve(\mathcal{B}/P_I)
\end{align*}
\end{proposition}

\begin{proof}
\begin{align*}
 \CoNerve(\mathcal{B}/P_X) &= C(R_{\mathcal{B}/P_X}) = C(R^{op}_{\mathcal{B}/P_I})= \Nerve(\mathcal{B}/P_I)
\end{align*}
Similarly:
\begin{align*}
 \Nerve(\mathcal{B}/P_X) &= C(R^{op}_{\mathcal{B}/P_X}) = C(R_{\mathcal{B}/P_I})= \CoNerve(\mathcal{B}/P_I)
\end{align*}
\end{proof}

\begin{corollary}
All four persistent modules associated with the quotient ball covers are isomorphic:
$$ H_*\Nerve(\mathcal{B}/P_I) \cong H_*\CoNerve(\mathcal{B}/P_I) \cong H_*\Nerve(\mathcal{B}/P_X) \cong H_*\CoNerve(\mathcal{B}/P_X) $$
\end{corollary}

We now establish a direct comparison between the standard Ball cover $\mathcal{B}$ and the quotient Ball cover $\mathcal{B}/P_X$.
Recall the definition of the pullback cover at scale $r$:
$$ \mathcal{B}_P(r) = q_r^*(\mathcal{B}/P_X(r)) = \{ q_r^{-1}(q_r(B(x, r))) \}_{x \in X} $$
The element $U_x \in \mathcal{B}_P(r)$ corresponding to index $x$ consists of all points $y \in X$ whose cluster $[y]$ intersects the ball $B(x, r)$.

\begin{theorem}[Ball-Quotient Interleaving]
\label{thm:ball-quotient-interleaving}
The Ball cover $\mathcal{B}$ and the pullback of the quotient Ball cover $\mathcal{B}_P$ are $2$-approximated.
That is, there exist refinement maps:
\begin{align*}
         f_r &: \mathcal{B}(r) \to\mathcal{B}_P(r) \tag{1}\\
        g_r &: \mathcal{B}_P(r) \to \mathcal{B}(2r) \tag{2}
\end{align*}
\end{theorem}

\begin{proof}
$\mathcal{B}(r) \to \mathcal{B}_P(r)$:
Let $B(x, r) \in \mathcal{B}(r)$.
For any $y \in B(x, r)$, we have $q_r(y) \in q_r(B(x, r))$, which implies $y \in q_r^{-1}(q_r(B(x, r)))$.
Thus, the set in the ball cover is strictly contained in the corresponding set of the pullback cover.
The refinement map is the identity on indices.

$\mathcal{B}_P(r) \to \mathcal{B}(2r)$:
Let $U_x = q_r^{-1}(q_r(B(x, r))) \in \mathcal{B}_P(r)$.
We must show $U_x \subseteq B(x, 2r)$.
Let $y \in U_x$.
By definition, $q_r(y) \cap q_r(B(x, r)) \neq \emptyset$ in the quotient space.
This means the cluster $[y]$ intersects the ball $B(x, r)$.
Therefore, there exists a witness point $z \in [y]$ such that $z \in B(x, r)$, i.e., $d(z, x) \leq r$.

Since $P$ is a clique partition at scale $r$, we know that $\diam([y]) \leq r$.
Since $y, z \in [y]$, we have $d(y, z) \leq r$.
By the triangle inequality:
$$ d(y, x) \leq d(y, z) + d(z, x) \leq r + r = 2r $$
Thus, $y \in B(x, 2r)$.
This holds for all $y \in U_x$, so $U_x \subseteq B(x, 2r)$.
The refinement map is again the identity on indices.
\end{proof}

\begin{corollary}
If $P_X$ is a persistent clique partition on $X$ then there exists a $2$-approximation with the standard \v Cech filtration:
$$ H_*\CoNerve(\mathcal{B}/P_X) \xleftrightarrow{2} H_*\Nerve(\mathcal{B}) \quad \text{(\v Cech filtration)}$$
\end{corollary}

\begin{proof}
$$ H_*\Nerve(\mathcal{B}) \xleftrightarrow{2} H_*\Nerve(\mathcal{B}_P) $$
We connect this to the quotient cover using the property that the nerve of a pullback is isomorphic to the nerve of the quotient cover (Lemma \ref{lem:pullback_nerve}):
$$ \Nerve(\mathcal{B}_P) = \Nerve(\mathcal{B}/P_X) $$
Finally, applying the duality established in the previous section ($\Nerve(\mathcal{B}/P_X) \cong \CoNerve(\mathcal{B}/P_X)$ via Dowker duality), we obtain the result.
\end{proof}

\section{Persistent Partitions}
\label{section:partitions}

All previous results assumed a persistent partition $P: \poset \to \Cat{Part}(X)$ with the property that for every scale $r \ge 0$,
\begin{equation}
\tag{$\star$}
\label{eq:clique_property}
\diam(C) \leq r \quad \text{for all } C \in P(r).
\end{equation}
The discrete partition $P_0(r) = \bigl\{\{x\} : x \in X\bigr\}$ satisfies~\eqref{eq:clique_property} trivially but it is not a useful one since its quotient map is the identity.
Conversely we seek partitions that are \emph{maximal}: as coarse as the diameter bound allows at every scale.

\begin{definition}[Maximal Clique-Partition]
A partition $P = \{C_i\}$ of a metric space $(X,d)$ at scale $s$ is an $s$-clique-partition if $\diam(C_i) \le s$ for all $i$.
It is \textbf{maximal} if $\diam(C_i \cup C_j) > s$ for every pair of distinct blocks $C_i, C_j \in P$.
\end{definition}

Note that maximal clique-partitions need not consist of maximal cliques; it is the partition that is maximal. They are also not unique in general~\cite{marin2023enumerating}.

\subsection{Hierarchical Clustering and Complete Linkage}
The natural source of persistent partitions is hierarchical clustering; a standard method for multi-scale analysis which constructs a series of nested partitions with respect to some linkage function.
Given a finite metric space $(X, d)$, a \emph{linkage function} $\ell$ assigns to each pair of disjoint non-empty subsets $A, B \subseteq X$  a non-negative real number depending only on the pairwise distances:
$$
\ell(A, B) = f\bigl(\{d(x, y) : x \in A,\, y \in B\}\bigr).
$$
The standard examples are single linkage ($f=\min$), complete linkage ($f=\max$),
and average linkage ($f=$ mean).

Given a linkage $\ell$, hierarchical clustering on $(X, d)$ produces a sequence of partitions
$$
P_{0} \prec P_{1} \prec \cdots \prec P_{N},
$$
where $P_{0}=\{\{x\}\mid\forall x\in X\}$ is the discrete partition, $P_{N} = \{X\}$, and each $P_{k+1}$ is obtained from $P_{k}$ by merging the two blocks $A, B \in P_{k}$ that correspond to the minimum value of $\ell$ over $P_{k}$; the \emph{merge scale} is $r_{k+1} := \ell(A, B)$, with $r_0 := 0$.
The sequence determines a persistent partition $P : \poset \to \Cat{Part}(X)$ by $P(s) = P_{k}$ for
$r_k \leq s < r_{k+1}$. See~\cite{jain1988algorithms,carlsson10acharacterization} for the standard formulation.

When several pairs achieve the minimum simultaneously --- a situation we call a \emph{tie} --- one has to make a choice which pairs to merge.
Different choices generally produce different sequences.
We address this specifically in the Section \ref{subsec:conservative}.

We now show that complete linkage $\ell (A,B) =\max_{x\in A,y\in B}d(x,y)$ is the most appropriate for our framework: it satisfies the diameter bound, it is the smallest linkage that does so, and it produces maximal clique-partitions at every scale.

\begin{proposition}[Diameter bound for complete linkage]\label{prop:complete-bound}
Let $P$ be the persistent partition produced by hierarchical clustering with complete linkage on a finite metric space $(X, d)$, under any choice of tie-breaking.
Then for every $r \geq 0$ and every $C \in P(r)$, $\diam(C) \leq r$.
\end{proposition}

\begin{proof}
We prove $\diam(C) \leq r_k$ for every $C \in P_k$ by induction on $k$.
The base case $P_{0}$ is the singleton partition, so $\diam(C) = 0 \leq r_0 = 0$.

Assume the bound holds for $P_{k}$.
The partition $P_{k+1}$ differs from $P_{(k)}$ only in that two blocks $A, B \in P_{k}$ have been replaced by $A \cup B$, with $\ell(A, B) = r_{k+1}$.
For any $x, y \in A \cup B$:
\begin{itemize}
\item if both lie in $A$, then $d(x, y) \leq \diam(A) \leq r_k \leq r_{k+1}$;
\item if both lie in $B$, similarly $d(x, y) \leq r_{k+1}$;
\item if $x \in A$, $y \in B$, then
$d(x, y) \leq \max_{a \in A,\, b \in B} d(a, b) = \ell(A, B) = r_{k+1}$.
\end{itemize}
Hence $\diam(A \cup B) \leq r_{k+1}$.
All other blocks of $P_{k+1}$ are unchanged from $P_{k}$ and satisfy the bound by hypothesis.
\end{proof}



\begin{proposition}
\label{prop:complete_linkage_minimal}
Let $\ell$ be a linkage function whose clustering sequence on every finite metric space respects the diameter bound~\eqref{eq:clique_property} at every step.
Then for any finite metric space $(X,d)$ and any merge of $\ell$'s clustering, where blocks $A, B$ merge at scale $r = \ell(A,B)$,
\[
  \ell(A,B) \;\geq\; \diam(A \cup B).
\]
Complete linkage achieves this bound with equality at every merge: when
$A, B$ are blocks formed during complete linkage's clustering,
\[
  \ell_{\mathrm{complete}}(A,B) \;=\; \diam(A \cup B).
\]
In this sense, complete linkage merges as early as the diameter bound permits.
\end{proposition}

\begin{proof}
The first inequality is the diameter bound itself: when $A, B$ merge at scale
$r$, the diameter bound on the resulting partition forces $\diam(A \cup B) \leq r = \ell(A,B)$.

For the equality, suppose $A, B$ are blocks of complete linkage's clustering,
formed at some step. Each was produced by previous merges at scales no greater
than $\ell_{\mathrm{complete}}(A,B)$, so $\diam(A) \leq \ell_{\mathrm{complete}}(A,B)$ and
$\diam(B) \leq \ell_{\mathrm{complete}}(A,B)$. Therefore
\[
  \diam(A \cup B)
  \;=\; \max\!\bigl(\diam(A),\, \diam(B),\, \max_{a \in A,\, b \in B} d(a,b)\bigr)
  \;=\; \max_{a \in A,\, b \in B} d(a,b)
  \;=\; \ell_{\mathrm{complete}}(A,B). \qedhere
\]
\end{proof}

Any linkage smaller than complete linkage may produce clusters with unbounded diameter relative to the scale, the well-known \emph{chaining effect}~\cite{carlsson2010classifyingclusteringschemes, ros2019hierarchical}.
In our framework, unbounded diameter renders the approximation guarantees of the previous sections impossible.
Conversely, any linkage strictly larger than complete linkage delays merges unnecessarily.
Furthermore, complete linkage produces maximal partitions:

\begin{proposition}
\label{prop:complete_linkage_maximal}
The partition produced by complete linkage at scale $s$ is a maximal $s$-clique-partition.
\end{proposition}

\begin{proof}
Let $P_s = \{C_1, \dots, C_k\}$ be the partition at scale~$s$.
Each block $C_i$ was formed by a merge at some scale $r \leq s$, and by definition of complete linkage the diameter of each block $C_i$ satisfies $\diam(C_i) \leq r \leq s$, establishing the clique property.
For maximality, let $C_i, C_j \in P_s$ be distinct.
These blocks remain unmerged at scale~$s$, which means $\ell(C_i, C_j) > s$ (otherwise the pair $(C_i, C_j)$ would have been merged at or before scale~$s$).
Since $\ell(C_i, C_j) = \max_{x \in C_i,\, y \in C_j} d(x,y)$, we obtain $\diam(C_i \cup C_j) \geq \ell(C_i, C_j) > s$.
\end{proof}

\subsection{Conservative Complete Linkage}
\label{subsec:conservative}

Hierarchical clustering with complete linkage has a well-known shortcoming: when multiple pairs of blocks share the same minimal linkage value (ties), the algorithm must choose an order in which to process them, and different orderings can produce different outputs~\cite[Section~3.2.6]{jain1988algorithms}.
This is often dismissed in practice by assuming general position, but it prevents the assignment from being well-defined as a functor.

\citeauthor{carlsson10acharacterization} ~\cite{carlsson2010classifyingclusteringschemes,carlsson10acharacterization,carlsson2008persistentclusteringtheoremj} resolve this by allowing multiple merges at each step \emph{transitively}: two blocks $A,B$ merge at scale $s$ if there exists a chain of blocks $A=B_1,\dots ,B_k=B$ such that $\ell(B_i, B_{i+1})\leq s$.
This eliminates order-dependence but violates the diameter bound~\eqref{eq:clique_property}.

\begin{example}
\label{ex:line}
Let $X = \{A,B,C\}$ with $d(A,B) = 1$, $d(B,C) = 1$, and $d(A,C) = 2$.
\begin{equation}
\label{eq:line_example}
\begin{tikzpicture}[
    point/.style={circle, fill=black, inner sep=1.5pt},
    connection/.style={thick, opacity=0.8},
    scale=0.4, baseline=(B.base)
]
    \node[point, label=below:A] (A) at (0,0) {};
    \node[point, label=below:B] (B) at (3,0) {};
    \node[point, label=below:C] (C) at (6,0) {};
    \draw[connection] (A) -- (B);
    \draw[connection] (B) -- (C);
    \node[right, align=left] at (8, 0) {
        $d(A,B) = 1$ \\
        $d(B,C) = 1$ \\
        $d(A,C) = 2$
    };
\end{tikzpicture}
\end{equation}
At scale $r=1$:
\begin{itemize}
    \item Hierarchical clustering with standard complete linkage could merge $\{A,B\}$ or $\{B,C\}$ resulting in two possible partitions: $\{\{A,B\}, \{C\}\}$ or $\{\{A\}, \{B,C\}\}$. Diameter $\le 1$ holds in both cases.
    \item \citeauthor{carlsson2010classifyingclusteringschemes}'s method merges $\{A,B\}$ and $\{B,C\}$ simultaneously via transitivity. Result: $P(s=1)=\{\{A,B,C\}\}$ with diameter $\diam(\{A,B,C\}) = 2 > 1$
\end{itemize}

While effective at solving the order dependence on the presence of ties, it does so at the cost of losing the fundamental diameter property~\eqref{eq:clique_property}.
\end{example}

We introduce \textbf{conservative complete linkage}, which resolves both issues: it is order-independent and respects the diameter bound at every scale.
Given a set of tied blocks that is not itself a clique under complete linkage, a full merge breaks the diameter bound and any partial merge reintroduces order-dependence.
Refusing to merge the component at all is the unique choice that preserves both the diameter bound and order-independence; the merge is postponed to the scale at which the component becomes a clique.

\begin{figure}[!htb]
    \centering
    \begin{subfigure}[b]{0.32\textwidth}
        \centering
        \begin{tikzpicture}[scale=1.5, thick]
            \draw[->, gray!50] (-0.2,0) -- (2.2,0);
            \draw[->, gray!50] (-0.2,0) -- (-0.2,2.5) node[left] {$d$};
            \foreach \y in {1,2} \draw[gray!30, dashed] (-0.2,\y) -- (2.2,\y);
            \node[left, font=\footnotesize] at (-0.2,1) {1};
            \node[left, font=\footnotesize] at (-0.2,2) {2};
            \draw (0,0) node[below]{A} -- (0,1);
            \draw (1,0) node[below]{B} -- (1,1);
            \draw (2,0) node[below]{C} -- (2,2);
            \draw (0,1) -- (1,1);
            \draw (0.5,1) -- (0.5,2);
            \draw (0.5,2) -- (2,2);
            \draw (1.25, 2) -- (1.25, 2.2);
        \end{tikzpicture}
        \caption{Standard complete linkage}
    \end{subfigure}
    \hfill
    \begin{subfigure}[b]{0.32\textwidth}
        \centering
        \begin{tikzpicture}[scale=1.5, thick]
            \draw[->, gray!50] (-0.2,0) -- (2.2,0);
            \draw[->, gray!50] (-0.2,0) -- (-0.2,2.5);
            \foreach \y in {1,2} \draw[gray!30, dashed] (-0.2,\y) -- (2.2,\y);
            \node[left, font=\footnotesize] at (-0.2,1) {1};
            \node[left, font=\footnotesize] at (-0.2,2) {2};
            \draw (0,0) node[below]{A} -- (0,1);
            \draw (1,0) node[below]{B} -- (1,1);
            \draw (2,0) node[below]{C} -- (2,1);
            \draw (0,1) -- (2,1);
            \draw (1,1) -- (1,2.2);
        \end{tikzpicture}
        \caption{\citeauthor{carlsson10acharacterization}}
    \end{subfigure}
    \hfill
    \begin{subfigure}[b]{0.32\textwidth}
        \centering
        \begin{tikzpicture}[scale=1.5, thick]
            \draw[->, gray!50] (-0.2,0) -- (2.2,0);
            \draw[->, gray!50] (-0.2,0) -- (-0.2,2.5);
            \foreach \y in {1,2} \draw[gray!30, dashed] (-0.2,\y) -- (2.2,\y);
            \node[left, font=\footnotesize] at (-0.2,1) {1};
            \node[left, font=\footnotesize] at (-0.2,2) {2};
            \draw (0,0) node[below]{A} -- (0,2);
            \draw (1,0) node[below]{B} -- (1,2);
            \draw (2,0) node[below]{C} -- (2,2);
            \draw (0,2) -- (2,2);
            \draw (1,2) -- (1,2.2);
        \end{tikzpicture}
        \caption{Conservative (Ours)}
    \end{subfigure}
    \caption{Dendrograms for Example~\ref{ex:line}.
    \textbf{(a)}~Standard linkage merges an arbitrary pair at scale~1.
    \textbf{(b)}~Transitive linkage merges all three at scale~1, violating the diameter bound.
    \textbf{(c)}~Conservative linkage postpones the merge to scale~2, where the all-pairs condition is satisfied.}
    \label{fig:linkage-comparison}
\end{figure}

\begin{definition}[Linkage Graph]
\label{def:linkage_graph}
Let $Q$ be a partition of a finite metric space $(X,d)$ and let $r \ge 0$.
The \textbf{linkage graph} $\Gamma_r(Q)$ has vertex set $Q$, with an edge between distinct blocks $B, B' \in Q$ if and only if $\ell(B, B') \leq r$.
\end{definition}

\begin{definition}[Conservative Complete Linkage]
\label{def:conservative}
Let $(X,d)$ be a finite metric space with complete linkage $\ell$, and let $0 < r_1 < r_2 < \cdots < r_m$ be the distinct values of~$d$.
The \textbf{conservative complete linkage partition} $P^* : \poset \to \Cat{Part}(X)$ is defined inductively:
\begin{itemize}
    \item $P^*(r) = \bigl\{\{x\} : x \in X\bigr\}$ for $r < r_1$.
    \item At each critical value $r_i$, let $Q = P^*(r_{i-1})$. For each connected component $C$ of the linkage graph $\Gamma_{r_i}(Q)$:
    \begin{itemize}
        \item if $C$ is a clique in $\Gamma_{r_i}(Q)$, replace the blocks of $C$ in $Q$ by the single block $\bigcup_{B \in C} B$;
        \item otherwise, leave the blocks of $C$ unchanged.
    \end{itemize}
    The resulting partition is $P^*(r_i)$.
    \item $P^*(r) = P^*(r_i)$ for $r \in [r_i, r_{i+1})$.
\end{itemize}
\end{definition}

In other words, groups of blocks that are merely connected components on the linkage graph are not merged; only sets of blocks that form cliques in the linkage graph are merged.
Returning to Example~\ref{ex:line}: at $r = 1$ the graph $\Gamma_1(\{\{A\},\{B\},\{C\}\})$ has edges $A$–$B$ and $B$–$C$, forming a single connected component $\{\{A\},\{B\},\{C\}\}$.
This component is not a clique since the edge $A$–$C$ is absent ($\ell(\{A\},\{C\}) = 2 > 1$), so no merge occurs.
At $r = 2$ the missing edge appears, the component is a clique, and all three blocks merge. See Figure~\ref{fig:linkage-comparison}.

We now verify the required properties.

\begin{proposition}[Diameter Bound]
\label{prop:conservative_diameter}
For every $r \geq 0$ and every block $C \in P^*(r)$, $\diam(C) \leq r$.
\end{proposition}
\begin{proof}
The argument is the inductive structure of Proposition~\ref{prop:complete-bound}.
The base case and within-block case are identical. 
The only difference is that the merge at step $i$ may now combine $m \geq 2$ blocks $B_1, \ldots, B_m$ rather than a single pair.
By construction, these blocks form a clique in $\Gamma_{r_i}(P^*(r_{i-1}))$, so $\ell(B_j, B_k) \leq r_i$ for all $j \neq k$.
For $x \in B_j$, $y \in B_k$
with $j \neq k$,
$$
d(x, y) \leq \max_{a \in B_j,\, b \in B_k} d(a, b) = \ell(B_j, B_k) \leq r_i.
$$
Combined with the inductive bound on each $B_j$, this gives $\mathrm{diam}\bigl(\bigcup_j B_j\bigr) \leq r_i$.
\end{proof}

\begin{proposition}
\label{prop:conservative_functor}
The assignment $P^*$ of Definition~\ref{def:conservative} defines a functor $P^*: \poset \to \Cat{Part}(X)$ and, together with Proposition~\ref{prop:conservative_diameter}, a persistent clique partition in the sense of Definition~\ref{def:persistent_clique_partition}.
\end{proposition}
\begin{proof}
\emph{Objects.} The assignment $r \mapsto P^*(r)$ is well-defined by Definition~\ref{def:conservative}.

\emph{Morphisms.} For consecutive critical values $r_{i-1} \leq r_i$, the partition $P^*(r_i)$ is obtained from $P^*(r_{i-1})$ by replacing certain collections of blocks $A_1, \dots, A_k$ by their union $C = A_1 \cup \dots \cup A_k$, and leaving every other block unchanged.
Define the refinement map $f : P^*(r_{i-1}) \to P^*(r_i)$ by $f(A_j) = C$ for each $j$ and $f(B) = B$ otherwise.
This map is the unique refinement: given any block $B \in P^*(r_{i-1})$, since blocks of $P^*(r_i)$ are pairwise disjoint, at most one of them contains $B$, and the construction guarantees that exactly one does.
The general case $P^*(r) \to P^*(s)$ for $r \leq s$ follows by composing the refinements at the critical values in between.

\emph{Identity and composition.} Both axioms follow from uniqueness: the identity $P^*(r) \to P^*(r)$ and any composite $P^*(s \leq t) \circ P^*(r \leq s)$ are refinement maps between the same pair of partitions, hence equal to the unique such map.
\end{proof}

In general position (all pairwise distances distinct), standard complete linkage, the transitive variant, and conservative complete linkage all produce the same output.
The three methods differ only in how they handle ties.
Pseudocode for the conservative algorithm is given in Appendix~\ref{sec:algos}; it takes as input a sorted edge list and runs in $O(N^2 \log N)$ time.
\section{Computation}
\label{sec:computation}

To compute persistence in practice we want to use state-of-the-art software such as Ripser \cite{Bauer2021Ripser}, and this imposes two requirements on the input: it must be a flag complex (so that the entire filtration is encoded by its 1-skeleton) and it must be a filtration (a nested sequence of inclusions, not a tower with contractions).
Neither holds for $\CoNerve(\mathcal{M}/\mathcal{P})$ out of the box: it is not a flag complex, and the sequence of quotient complexes is a simplicial tower, since cluster merges induce vertex contractions rather than inclusions.

Both gaps are bridged naturally within our framework.
For the first, although the previous sections developed the theory in a metric setting, because that is where the Vietoris–Rips complex and maximal-clique cover are classically defined, neither construction actually invokes the triangle inequality: both are determined by a diameter condition on a symmetric, reflexive pairwise dissimilarity.
The 1-skeleton of $\CoNerve(\mathcal{M}/\mathcal{P})$ defines  exactly such a dissimilarity $d_\mathcal{P}$ on $X/\mathcal{P}$, and its flag completion is the Vietoris–Rips complex of $(X/\mathcal{P}, d_\mathcal{P})$.
The cover-level approximation guarantees of Section~\ref{sec:quotient} transfer to this flag complex unchanged: they require $d$ to be a metric, but not $d_P$.
For the second, we convert the resulting tower into a filtration via the coning procedure of ~\citeauthor{kerber2019barcodes}~\cite{kerber2019barcodes}, which increases the complex size only marginally and preserves the persistence barcode.
Both adaptations operate within the cover framework and preserve the approximation guarantees established in the previous sections.

\subsection{Constructing a Flag Complex}
The Vietoris–Rips complex $\mathrm{Rips}(Y,d)(r)=\{\sigma\subseteq Y:\diam(\sigma)_d\leq r\}$ and the maximal-clique cover do not require the triangle inequality and can be generalized to the quotient space where we have a symmetric dissmilarity measure (a metric absent of triangle inequality) defined by the $1$-skeleton of  $\CoNerve(\mathcal{M}/\mathcal{P})(r)$:

\begin{lemma}
\label{lem:quotient_metric}
An edge $\{C_i, C_j\}$ exists in $\CoNerve(\mathcal{M}/\mathcal{P})(r)$ if and only if some maximal $r$-clique $M$ intersects both $C_i$ and $C_j$, which holds precisely when
\[
    d_{\mathcal{P}(r)}(C_i, C_j) \;:=\; \min_{x \in C_i,\, y \in C_j} d(x, y) \;\leq\; r.
\]
\end{lemma}
We write $\mathcal{M}_Q(r)$ for the maximal clique cover of $(X/\mathcal{P}(r),\, d_{\mathcal{P}(r)})$ and observe:
\[
    \mathrm{Flag}_{\mathrm{1\text{-}skel}}(\CoNerve(\mathcal{M}/\mathcal{P})(r))
    \;=\; \mathrm{Rips}(X/\mathcal{P},\, d_\mathcal{P})(r)
    \;=\; \CoNerve(\mathcal{M}_Q)(r).
\]
The flag completion is a strict enlargement:
$\CoNerve(\mathcal{M}/\mathcal{P})(r) \;\subseteq\; \CoNerve(\mathcal{M}_Q)(r)$, filling in high-dimensional simplices whenever all pairwise connections exist.
Crucially, this is the same functor $\CoNerve \circ \mathcal{M}$ applied to a \emph{smaller} (non-metric) space.
The approximation guarantee transfers by the same mechanism as before.


\begin{lemma}\label{lem:pullback-MQ-persistent}
The pullback cover $\mathcal{M}_Q^{\mathcal{P}}(r) = q_r^*(\mathcal{M}_Q(r))$
is a persistent cover of $X$.
\end{lemma}

\begin{proof}
Let $r \leq s$. Since $\mathcal{P}(r)$ refines $\mathcal{P}(s)$, the induced map $\pi_{r,s} : X/\mathcal{P}(r) \to X/\mathcal{P}(s)$ satisfies: $d_{\mathcal{P}(s)}(\pi(C_i), \pi(C_j)) \leq d_{\mathcal{P}(r)}(C_i, C_j)$, since the minimum on the right is taken over a subset of the points considered on the left.
Hence any $r$-clique in $(X/\mathcal{P}(r), d_{\mathcal{P}(r)})$ maps to an $s$-clique in $(X/\mathcal{P}(s), d_{\mathcal{P}(s)})$, which by Lemma~\ref{lem:maximal_clique_refinement} is contained in some maximal $s$-clique.
The persistence of the pullback then follows from Proposition~\ref{prop:pullback-persistence}.
\end{proof}

\begin{proposition}\label{prop:flag-approx}
Let $\mathcal{P}$ be a persistent clique partition of $(X, d)$, and let $\mathcal{M}_Q$ denote the maximal clique cover of the quotient space
$(X/\mathcal{P},\, d_{\mathcal{P}})$.
Then
\[
    \mathcal{M} \xleftrightarrow{3} \mathcal{M}_Q^{\mathcal{P}},
\]
where $\mathcal{M}_Q^{\mathcal{P}}(r) = q_r^*(\mathcal{M}_Q(r))$ is the pullback
cover on $X$.
\end{proposition}

\begin{proof}
We construct the two families of refinement maps.

\medskip\noindent\textbf{Forward:} $\mathcal{M}(r) \to \mathcal{M}_Q^{\mathcal{P}}(r)$.
Let $M \in \mathcal{M}(r)$ be a maximal $r$-clique in $X$. For any two partition blocks
$C_i, C_j$ intersecting $M$, there exist $x \in C_i \cap M$ and $y \in C_j \cap M$ with $d(x,y) \leq r$, so $d_{\mathcal{P}}(C_i, C_j) \leq r$.
Thus $q_r(M)$ is an $r$-clique in the quotient space and is contained in some maximal $r$-clique $K \in \mathcal{M}_Q(r)$.
It follows that $M \subseteq q_r^{-1}(q_r(M)) \subseteq q_r^{-1}(K) \in \mathcal{M}_Q^{\mathcal{P}}(r)$.

\medskip\noindent\textbf{Backward:} $\mathcal{M}_Q^{\mathcal{P}}(r) \to \mathcal{M}(3r)$.
Let $U = q_r^{-1}(K) \in \mathcal{M}_Q^{\mathcal{P}}(r)$ for some maximal $r$-clique $K$ in the quotient space.
We bound the diameter of $U$. Let $a, b \in U$, belonging to partition blocks $C_a$ and $C_b$ respectively, with $C_a, C_b \in K$.
Since $K$ is an $r$-clique, $d_{\mathcal{P}}(C_a, C_b) \leq r$, so there exist $u \in C_a,\, v \in C_b$ with $d(u,v) \leq r$.
By the triangle inequality:
\[
    d(a, b) \;\leq\; \underbrace{d(a, u)}_{\leq\, r} + \underbrace{d(u, v)}_{\leq\, r}
    + \underbrace{d(v, b)}_{\leq\, r} \;\leq\; 3r,
\]
where the first and third terms use the clique-partition diameter bound $\mathrm{diam}(C_a) \leq r$ and $\mathrm{diam}(C_b) \leq r$.
Thus $\mathrm{diam}(U) \leq 3r$, and by Proposition~\ref{prop:clique_properties}, $U$ is contained in some maximal $3r$-clique.
\end{proof}

Note that the triangle inequality is invoked only on $(X, d)$ where it holds; the dissimilarity $d_\mathcal{P}$ on the quotient is never required to satisfy it.

\begin{corollary}\label{cor:flag-interleaving}
Let $\mathcal{P}$ be a persistent clique partition.
The persistence module of the Vietoris-Rips filtration of the quotient space is $3$-approximated with the standard Vietoris-Rips filtration:
\[
    H_*\CoNerve(\mathcal{M}) \xleftrightarrow{3} H_*\CoNerve(\mathcal{M}_Q).
\]
\end{corollary}

\begin{proof}
By Proposition~\ref{prop:flag-approx} and Theorem~\ref{thm:main_propagation}, we obtain
$H_*\CoNerve(\mathcal{M}) \xleftrightarrow{3} H_*\CoNerve(\mathcal{M}_Q^{\mathcal{P}})$.
The chain of natural isomorphisms
$$
    H_*\CoNerve(\mathcal{M}_Q^{\mathcal{P}})
    \;\cong\; H_*\mathrm{Nrv}(\mathcal{M}_Q^{\mathcal{P}})
    \;\cong\; H_*\mathrm{Nrv}(\mathcal{M}_Q)
    \;\cong\; H_*\CoNerve(\mathcal{M}_Q)
$$
follows from Dowker duality and Lemma~\ref{lem:pullback_nerve} (the pullback preserves the nerve).
\end{proof}

Since $\CoNerve(\mathcal{M}_Q)$ is a flag complex, it is entirely determined by its $1$-skeleton, all operations like edge insertions and cluster merges, can be performed on a weighted graph.
This graph is passed to Ripser (or other efficient persistent homology software).

\subsection{From Towers to Filtrations}\label{subsec:towers-to-filtrations}

The sequence of quotient complexes $\CoNerve(\mathcal{M}_Q)(r)$ does not form a filtration in the classical sense.
As the parameter $r$ increases, two competing forces act on the complex: new edges appear (as pairwise distances fall below the threshold), but cluster merges \emph{contract} vertices, producing simplicial maps that are not inclusions.
The result is a \emph{simplicial tower}: a sequence of simplicial complexes connected by simplicial maps that include both inclusions and contractions.

Since the complex $\CoNerve(\mathcal{M}_Q)$ is a flag complex, it is entirely determined by its $1$-skeleton. In particular, every contraction in the tower is an \emph{edge contraction} in the underlying graph: when two clusters $C_i$ and $C_j$ merge at scale $c$, the vertex $C_j$ is identified with $C_i$ in the $1$-skeleton.
This is the elementary contraction $(C_i, C_j) \mapsto C_i$ in the sense of Dey et al.~\cite{dey2014computing}.

We convert the resulting simplicial tower into a filtration using the \emph{coning} technique~\cite{kerber2019barcodes,dey2014computing}, which
replaces each contraction with a sequence of simplex inclusions that preserve the barcode.
Each contraction $(u, v) \mapsto u$ in a complex $K$ is replaced with the inclusion of the cone $u * \overline{\mathrm{St}}(v, K)$ (Fig. ~\ref{fig:conning}).
The key optimization of~\cite{kerber2019barcodes} is to cone only over the \emph{active} closed star and to choose the representative (between $u$ and $v$) that minimizes the number of added simplices.

Given that we have a flag complex we apply the coning procedure on the $1$-skeleton: coning a vertex $v$ into $u$ adds edges from $u$ to every active neighbor of $v$ that is not already adjacent to $u$.
The resulting weighted graph encodes the $1$-skeleton of the equivalent filtration.
\begin{figure}[htb]
    \centering
    \begin{tikzpicture}[
  triangle/.style       = {fill=color_r230g57b71, opacity=0.3, draw=none},
  cone_triangle/.style  = {fill=color_r230g57b71, opacity=0.15, draw=none},
  edge_orig/.style      = {black, opacity=0.8, thick},
  node/.style           = {circle, fill=black, inner sep=1.5pt},
  scale=0.7
]
  \definecolor{color_r230g57b71}{rgb}{0.902,0.224,0.278}

  %

  \begin{scope}[xshift=0cm]
    \coordinate (u)   at ( 0.0, 2.5);
    \coordinate (v)   at ( 0.6, 1.0);
    \coordinate (t1)  at (-0.7, 3.3);
    \coordinate (t2)  at ( 0.8, 3.4);
    \coordinate (p)   at (-2.5, 3.0);
    \coordinate (sq1) at (-1.1, 0.7);
    \coordinate (sq2) at (-0.4,-0.9);
    \coordinate (sq3) at ( 1.0,-1.0);
    \coordinate (r1)  at ( 2.0, 1.7);
    \coordinate (r2)  at ( 2.6, 0.6);

    \fill[triangle] (u)  -- (t1)  -- (t2)  -- cycle;
    \fill[triangle] (v)  -- (sq1) -- (sq2) -- cycle;
    \fill[triangle] (v)  -- (sq2) -- (sq3) -- cycle;
    \fill[triangle] (v)  -- (r1)  -- (r2)  -- cycle;

    \draw[edge_orig] (t1) -- (p);   
    \draw[edge_orig] (u)  -- (t1);
    \draw[edge_orig] (u)  -- (t2);
    \draw[edge_orig] (t1) -- (t2);
    \draw[edge_orig] (u)  -- (v);
    \draw[edge_orig] (v)  -- (sq1);
    \draw[edge_orig] (sq1)-- (sq2);
    \draw[edge_orig] (v)  -- (sq2);
    \draw[edge_orig] (sq2)-- (sq3);
    \draw[edge_orig] (v)  -- (sq3);
    \draw[edge_orig] (v)  -- (r1);
    \draw[edge_orig] (r1) -- (r2);
    \draw[edge_orig] (v)  -- (r2);

    \foreach \pt in {u,v,t1,t2,p,sq1,sq2,sq3,r1,r2}
      \node[node] at (\pt) {};

    \node[below left=1pt and 1pt]  at (u) {$u$};
    \node[above left=1pt and 3pt] at (v) {$v$};
  \end{scope}


\draw[right hook->, thick, >=stealth]
  (3.7, 1.0) -- node[above, midway, font=\small] {coning} (5.4, 1.0);

  \begin{scope}[xshift=8.9cm]
    \coordinate (u)   at ( 0.0, 2.5);
    \coordinate (v)   at ( 0.6, 1.0);
    \coordinate (t1)  at (-0.7, 3.3);
    \coordinate (t2)  at ( 0.8, 3.4);
    \coordinate (p)   at (-2.5, 3.0);
    \coordinate (sq1) at (-1.1, 0.7);
    \coordinate (sq2) at (-0.4,-0.9);
    \coordinate (sq3) at ( 1.0,-1.0);
    \coordinate (r1)  at ( 2.0, 1.7);
    \coordinate (r2)  at ( 2.6, 0.6);

    \fill[cone_triangle] (u) -- (v)  -- (t1) -- cycle;
    \fill[cone_triangle] (u) -- (v)  -- (t2) -- cycle;
    \fill[cone_triangle] (v) -- (t1) -- (t2) -- cycle;

    \fill[triangle] (u)  -- (t1)  -- (t2)  -- cycle;
    \fill[triangle] (v)  -- (sq1) -- (sq2) -- cycle;
    \fill[triangle] (v)  -- (sq2) -- (sq3) -- cycle;
    \fill[triangle] (v)  -- (r1)  -- (r2)  -- cycle;

    \draw[edge_orig] (t1) -- (p);
    \draw[edge_orig] (u)  -- (t1);
    \draw[edge_orig] (u)  -- (t2);
    \draw[edge_orig] (t1) -- (t2);
    \draw[edge_orig] (u)  -- (v);
    \draw[edge_orig] (v)  -- (sq1);
    \draw[edge_orig] (sq1)-- (sq2);
    \draw[edge_orig] (v)  -- (sq2);
    \draw[edge_orig] (sq2)-- (sq3);
    \draw[edge_orig] (v)  -- (sq3);
    \draw[edge_orig] (v)  -- (r1);
    \draw[edge_orig] (r1) -- (r2);
    \draw[edge_orig] (v)  -- (r2);

    \draw[edge_orig] (v) -- (t1);
    \draw[edge_orig] (v) -- (t2);

    \foreach \pt in {u,v,t1,t2,p,sq1,sq2,sq3,r1,r2}
      \node[node] at (\pt) {};

    \node[below left=1pt and 1pt]  at (u) {$u$};
    \node[above left=1pt and 3pt] at (v) {$v$};

  \end{scope}

\end{tikzpicture}
    \caption{Coning replaces the contraction $(u, v) \mapsto u$ with the inclusion of the cone $u * \overline{\mathrm{St}}(v, K)$ (right), yielding an equivalent filtration that preserves the original barcode. To minimize added simplices, we apply the optimization from~\cite{kerber2019barcodes} and cone over the vertex with the smaller closed star: $u * \overline{\mathrm{St}}(v, K)$ vs $v * \overline{\mathrm{St}}(u, K)$.}
    \label{fig:conning}
\end{figure}

\subsection{Benchmarks}

\paragraph{Datasets} We consider both real-world biological datasets and synthetic ones. 
Starting with the real datasets: \texttt{pbmc3k} is a single-cell RNA sequencing dataset comprising 3,000 human peripheral blood mononuclear cells; the raw gene expression count matrix is preprocessed and reduced to 50 dimensions following \cite{satija2015spatial,wolf2018scanpy}. 
We also consider three biological datasets from the benchmark \cite{otter2017roadmap}: \texttt{celegans}, representing the neural network of \textit{C. elegans}, the \texttt{H1V1}, consisting of genomic sequences from the HIV1 virus and the \texttt{vicsek} biological model for particle aggregation.  

We also consider the following synthetic datasets.
We generated samples from the Klein bottle (\texttt{klein\_400}, \texttt{klein\_900}) and use the \textit{Stanford Dragon Dataset} (\texttt{dragon\_1k}, \texttt{dragon\_2k}) taken from \cite{otter2017roadmap}. 
To test higher-dimensional properties, we include random orthogonal $3\times 3$ matrices (\texttt{o3\_1024}, \texttt{o3\_2048}) from \cite{Bauer2021Ripser,bauer2017phat}. 
Finally, we generated random samples from the unit $n$-sphere in $\mathbb{R}^{n+1}$ and a torus in $\mathbb{R}^3$ using the \texttt{TaDAsets} Python package.

\paragraph{Setup}
All benchmarks were run on a M1 Macbook Pro with 16 Gbs of memory.
Runtime was capped at 24h, for all datasets we consider the threshold given by the minimal enclosing radius $\min_{x\in X} \max_{y\in X} d(x,y)$.
For Table \ref{tab:filtration_comparison} we computed up to $H_1$ homology while for Figure \ref{fig:benchmarks}a up to $H_2$.
Naturally, for Figure \ref{fig:benchmarks}b, barcodes were computed up to $H_n$ for each $n-$sphere considered.
We use Ripser \cite{Bauer2021Ripser} for both the filtrations resulting from our quotient covers and the standard Vietoris-Rips.
Due to integer overflow in Ripser's combinatorial number system we had to use the 128bit version for $H_5$.

For each dataset we computed the $1-$skeleton of the equivalent filtration of the tower (Algorithm~\ref{algo:main}).
We have the pseudocode for all algorithms in the Appendix~\ref{sec:algos}.
The code and python package is available at \url{https://github.com/antonio-leitao/coperto}.

\paragraph{Results}

Beyond the reduction in filtration size, consistently reaching 3 to 4 orders of magnitude (Table \ref{tab:filtration_comparison}, Fig. \ref{fig:results_complexity}), we observe that our method is most effective precisely when the Vietoris-Rips filtration struggles.
Figure \ref{fig:results_complexity} illustrates this relationship between dataset complexity and our efficiency gains.
The $x$-axis measures the number of Vietoris-Rips simplices generated per data point.
For datasets with non-trivial topological structure, such as \texttt{dragon}, \texttt{klein}, and \texttt{o3}, the standard Vietoris-Rips filtration explodes in size.
In contrast, our method compresses this redundancy while still capturing the topology.
These topologically rich datasets are precisely the ones where we observe the largest reduction in filtration size ($y-$axis) and runtime speedup (marker size).
This effectiveness at describing topology with very few simplices enables us to capture high dimensional homology very efficiently (Fig. \ref{fig:benchmarks_sphere}, \ref{fig:benchmarks_h5})

Furthermore, our approach restores linear scalability to the filtration construction.
Figure \ref{fig:results_scalability} compares the empirical growth rates of the standard VR filtration against our method across dataset pairs (e.g., \texttt{dragon\_1k} vs. \texttt{dragon\_2k}).
We estimate the scaling complexity $O(N^\alpha)$ by comparing the growth of the filtration size ($S$) relative to the number of points ($N$) across dataset pairs (e.g., \texttt{dragon\_1k} vs. \texttt{dragon\_2k}).
The growth exponent is calculated as $\alpha = \frac{\ln(S_2) - \ln(S_1)}{\ln(N_2) - \ln(N_1)}$.
For Vietoris-Rips filtrations we observe a value of $\alpha \approx 3.0$ indicates cubic scaling where doubling the input size results in an $8\times$ increase in filtration size ($2^3$), which is to be expected for a clique filtration computed up to triangles.
Conversely, our filtration size roughly doubles, exhibiting near-linear scalability ($\alpha \approx 1.0$) where doubling the input results in only a $\approx 2\times$ increase in filtration size.
By restoring linear scalability, our method effectively transforms Persistent Homology computation from an exponential problem into a linear one, making it feasible for significantly larger datasets (Fig.~\ref{fig:benchmarks_torus}).

\begin{figure*}[!htb]

    \centering
    \begin{subfigure}[b]{0.45\textwidth}
        \centering
        \includegraphics[height=5cm]{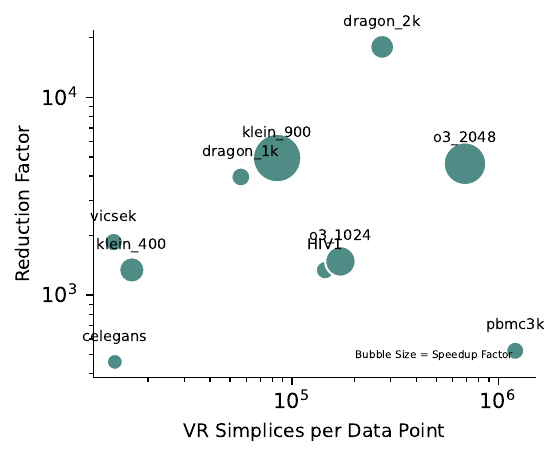}
        \caption{}
    \label{fig:results_complexity}
    \end{subfigure}
    \begin{subfigure}[b]{0.45\textwidth}
        \centering
        \includegraphics[height=5cm]{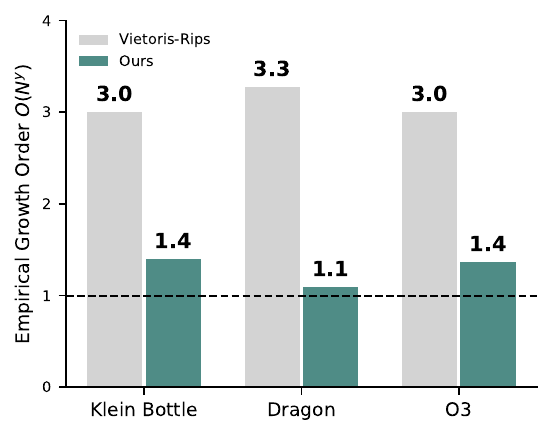}
        \caption{}
    \label{fig:results_scalability}
    \end{subfigure}
    \caption{
    \textbf{(a)}~Filtration size reduction factor ($y$-axis) and subsequent computation speedup (marker size) as a function of the number of simplices in the Vietoris-Rips filtration per data point ($x-$axis).
    \textbf{(b)}~The Empirical Order of Growth of the filtration size ($S$) relative to the number of points ($N$).
    The $y-$value represents the empirical exponent.
    }
    \label{fig:results}
\end{figure*}

\section{Conclusion}

The central message of this work is that covers, not simplicial complexes, are the natural level of abstraction for constructing and comparing filtrations.
All diagrams of refinement maps commute up to contiguity automatically, existence alone is sufficient.
Any functor from covers to simplicial complexes that preserves contiguity then maps these diagrams to diagrams commuting up to chain homotopy, from which strict algebraic interleavings follow.
The Nerve and Co-Nerve being two such functors.
In this sense, the interleaving guarantees between persistence modules originate at the cover level.
This means that any new functor from covers to simplicial complexes satisfying this single property would inherit the full cover-level machinery: the automatic commutativity, the interleaving guarantees, and the quotient constructions developed here.
Identifying such functors, and understanding whether they capture topological information beyond what the Nerve and Co-Nerve provide, is a natural direction.

As a first application of this framework, we constructed an unconditional approximation of the Vietoris–Rips filtration by quotienting covers with persistent partitions built from conservative complete linkage.
The resulting filtrations are orders of magnitude smaller while maintaining a $\log 3$-interleaving, restoring empirical near-linear scaling in the number of data points and making both large datasets and high homological dimensions accessible.
The approximation constant of $3$ shows up from the triangle inequality applied across a single partition block; tighter constants may be achievable with partitions satisfying stronger geometric constraints, and our experiments suggest that the empirical interleaving is often much tighter than this worst-case bound.

\section*{Acknowledgements}
I thank Nina Otter for guidance and encouragement throughout this work.
\printbibliography
\appendix
\section{Algorithms}
\label{sec:algos}

This is the main algorithm that takes in a distance matrix, sorted, and returns a weighted graph representing a filtration of the $\CoNerve(\mathcal{M}_p)$ functor.
The algorithm runtime is dominated by the $\mathcal{O}(N^2\log N)$ of the distance matrix computation and sorting.
\begin{algorithm}[H]
\caption{Creating a Filtered 1-skeleton from a Tower}
\label{algo:main}
\begin{algorithmic}[1]
\Require Sorted edge list $E_{\text{sorted}} = \{(u,v,d)\}$
\Ensure Filtered 1-skeleton $G$
\State $C \gets$ new \textsc{ConservativeCompleteLinkage}($n$)
\State $T \gets$ new \textsc{Skeleton}($n$)
\State $i \gets 0$
\For{$i = 0$ \textbf{to} $|E_{\text{sorted}}|-1$}
    \State $(u,v,d) \gets E_{\text{sorted}}[i]$
    \State $C.\textsc{AddEdge}(u,v,d)$
    \State $T.\textsc{AddEdge}(u,v,d)$
    \State $d_{\text{next}} \gets$ $(i < |E_{\text{sorted}}|-1)$ ? $E_{\text{sorted}}[i+1].\text{dist}$ : $\bot$ \Comment{Peek next distance}
    \If{$d_{\text{next}} \ne d$}
        \For{ $(x,y,d_{\text{contr}})$ in $C.\textsc{Contractions}(d)$}
            \State $T.\textsc{Contract}(x,y,d_{\text{contr}})$
        \EndFor
    \EndIf
\EndFor
\Return $T.G$ \Comment{Final weighted graph}
\end{algorithmic}
\end{algorithm}

\begin{algorithm}[H]
\caption{Conservative Complete Linkage (Object)}
\label{alg:ccl-class}
\begin{algorithmic}[1]
\Class{\textsc{ConservativeCompleteLinkage}}
    \State $\mathit{uf}$: Union-Find on $n$ elements
    \State $\mathit{size}[\cdot]$: cluster sizes (initially $1$)
    \State $\mathit{edges}[\cdot][\cdot]$: inter-cluster edge counts (initially $0$)
    \Statex
    \Function{AddEdge}{$u, v, d$}
        \State $r_u \gets \mathit{uf}.\mathrm{find}(u)$;\quad $r_v \gets \mathit{uf}.\mathrm{find}(v)$
        \If{$r_u \neq r_v$}
            \State $\mathit{edges}[r_u][r_v] \mathrel{+}= 1$;\quad $\mathit{edges}[r_v][r_u] \mathrel{+}= 1$
        \EndIf
    \EndFunction
    \Statex
    \Function{Contractions}{$d$} \Comment{Call once per tie batch}
        \Repeat
            \State $\mathit{merged} \gets \textbf{false}$
            \State $\Gamma \gets$ graph on active roots: edge $(A,B)$ iff $\mathit{edges}[A][B] = \mathit{size}[A] \cdot \mathit{size}[B]$
            \For{each connected component $\mathcal{C}$ of $\Gamma$}
                \If{$\mathcal{C}$ is a clique in $\Gamma$}
                    \State Choose representative $w \in \mathcal{C}$
                    \For{each $B_k \in \mathcal{C} \setminus \{w\}$}
                        \State $\mathit{uf}.\mathrm{union}(w, B_k)$
                        \State $\mathit{size}[w] \gets \mathit{size}[w] + \mathit{size}[B_k]$
                        \State Aggregate $\mathit{edges}[B_k][\cdot]$ into $\mathit{edges}[w][\cdot]$
                        \State \textbf{yield} $(w, B_k, d)$
                    \EndFor
                    \State $\mathit{merged} \gets \textbf{true}$
                \EndIf
            \EndFor
        \Until{$\lnot\, \mathit{merged}$}
    \EndFunction
\EndClass
\end{algorithmic}
\end{algorithm}

\begin{algorithm}[H]
\caption{Creating a 1-Skeleton}
\label{alg:coning_oo}
\begin{algorithmic}[1]
\Class{\textsc{Skeleton}}
    \State \quad $G$: Weighted graph storing birth times $D[u][v]$ (initially $\infty$)
    \State \quad $\mathcal{A}$: Set of active vertices
    \State \quad $\mathcal{U}$: Union-Find data structure

    \Function{AddEdge}{$u, v, \text{dist}$}
        \State $u \gets \mathcal{U}.\text{find}(u)$; \ $v \gets \mathcal{U}.\text{find}(v)$
        \If{$u \neq v$ \textbf{and} $u, v \in \mathcal{A}$}
            \State $D[u][v] \gets \min(D[u][v], \text{dist})$ \Comment{Preserve original birth time}
        \EndIf
    \EndFunction

    \Function{Contract}{$u, v, \text{dist}$}
        \State $u \gets \mathcal{U}.\text{find}(u)$; \ $v \gets \mathcal{U}.\text{find}(v)$
        \If{$u = v$} \State \Return \EndIf

        \State $N_u \gets \text{Neighbors}(u) \cap \mathcal{A}$  \Comment{Compute active stars} 
        \State $N_v \gets \text{Neighbors}(v) \cap \mathcal{A}$

        \If{$|N_u| > |N_v|$}  \Comment{Cone the smaller star into the larger} 
            \State \textbf{swap} $u \leftrightarrow v$
            \State \textbf{swap} $N_u \leftrightarrow N_v$
        \EndIf
        \For{\textbf{each} $w \in N_u$ such that $w \neq v$}
             \State \textsc{AddEdge}($v, w, t_{\text{cone}}$)
        \EndFor
        
        \State $\mathcal{A}.\text{remove}(u)$
        \State $\mathcal{U}.\text{union}(u \to v)$ \Comment{$u$ points to $v$}
    \EndFunction
\EndClass
\end{algorithmic}
\end{algorithm}

\begin{algorithm}[H]
\caption{Complete Linkage Clustering}
\begin{algorithmic}[1]
\Require Number of points $n$, sorted edge list $E_{sorted} = \{(u,v,d)\}$
\Ensure Merge history $\mathcal{H}$

\State Initialize DSU with $n$ singleton clusters
\For{each cluster root $r$}
    \State $size[r] \gets 1$
    \State $edges[r] \gets \emptyset$ \Comment{Map: neighbor $\rightarrow$ edge count}
\EndFor
\State $\mathcal{H} \gets \emptyset$

\For{each $(u,v,d) \in E_{sorted}$}
    \State $r_u \gets \mathrm{find}(u)$
    \State $r_v \gets \mathrm{find}(v)$
    \If{$r_u = r_v$}
        \State \textbf{continue}
    \EndIf

    \State $edges[r_u][r_v] \gets edges[r_u][r_v] + 1$
    \State $edges[r_v][r_u] \gets edges[r_v][r_u] + 1$

    \If{$edges[r_u][r_v] = size[r_u] \cdot size[r_v]$}
        \Comment{Complete linkage condition satisfied}

        \State Choose winner $w \in \{r_u,r_v\}$ with larger $|edges[w]|$
        \State $l \gets$ the other root
        \State $\mathrm{union}(w,l)$

        \State $size[w] \gets size[r_u] + size[r_v]$

        \For{each neighbor $x$ of $l$}
            \State $edges[w][x] \gets edges[w][x] + edges[l][x]$
            \State $edges[x][w] \gets edges[w][x]$
            \State remove $edges[x][l]$
        \EndFor
        \State remove $edges[w][l]$

        \State Append $(r_u,r_v,d)$ to $\mathcal{H}$
    \EndIf
\EndFor

\Return $\mathcal{H}$
\end{algorithmic}
\end{algorithm}

\begin{algorithm}[H]
\caption{Conservative Complete Linkage}
\label{algo:conservative}
\begin{algorithmic}[1]
\Require Number of points $n$, sorted edge list $E_{\text{sorted}} = \{(u,v,d)\}$
\Ensure Merge history $\mathcal{H}$
\State Initialize DSU with $n$ singleton clusters
\For{each cluster root $r$}
    \State $\mathit{size}[r] \gets 1$
    \State $\mathit{edges}[r] \gets \emptyset$ \Comment{Map: neighbor $\to$ edge count}
\EndFor
\State $\mathcal{H} \gets \emptyset$

\For{$i = 0$ \textbf{to} $|E_{\text{sorted}}|-1$}
    \State $(u, v, d) \gets E_{\text{sorted}}[i]$
    \Statex
    \Comment{\textbf{Phase 1:} Add all tied edges}
    \State $r_u \gets \mathrm{find}(u),\; r_v \gets \mathrm{find}(v)$
    \If{$r_u \neq r_v$}
        \State $\mathit{edges}[r_u][r_v] \mathrel{+}= 1$;\quad $\mathit{edges}[r_v][r_u] \mathrel{+}= 1$
    \EndIf

    \State $d_{\text{next}} \gets (i < |E_{\text{sorted}}|-1)$ ?
            $E_{\text{sorted}}[i+1].\text{dist}$ : $\bot$

    \Statex
    \Comment{\textbf{Phase 2:} Merge clique components when batch of ties ends}
    \If{$d_{\text{next}} \ne d$}
        \Repeat
            \State $\mathit{merged} \gets \textbf{false}$
            \State $\Gamma \gets$ graph on active roots with edge $(A,B)$ iff $\mathit{edges}[A][B] = \mathit{size}[A] \cdot \mathit{size}[B]$
            \For{each connected component $\mathcal{C} = \{B_1,\dots,B_m\}$ of $\Gamma$}
                \If{$\mathcal{C}$ is a clique in $\Gamma$} \Comment{All $\binom{m}{2}$ pairs are edges}
                    \State Choose representative $w \in \mathcal{C}$
                    \For{each $B_k \in \mathcal{C} \setminus \{w\}$}
                        \State $\mathrm{union}(w, B_k)$
                        \State $\mathit{size}[w] \gets \mathit{size}[w] + \mathit{size}[B_k]$
                        \State Aggregate $\mathit{edges}[B_k][\cdot]$ into $\mathit{edges}[w][\cdot]$
                        \State Append $(w, B_k, d)$ to $\mathcal{H}$
                    \EndFor
                    \State $\mathit{merged} \gets \textbf{true}$
                \EndIf
            \EndFor
        \Until{$\lnot\, \mathit{merged}$}
    \EndIf
\EndFor

\Return $\mathcal{H}$
\end{algorithmic}
\end{algorithm}

\end{document}